\documentclass[12pt]{article}
\usepackage{amsfonts}
\usepackage{amsmath, amsthm, full page, tikz-cd, comment,
amssymb, amscd, graphicx, enumitem, stmaryrd, enumitem%, showkeys
}
\usepackage{xcolor}
\allowdisplaybreaks
\newtheorem{thm}{Theorem}[section]
\newtheorem{defn}[thm]{Definition}
\newtheorem{prop}[thm]{Proposition}
\newtheorem{cor}[thm]{Corollary}

\newtheorem{rema}[thm]{Remark}

\newcommand{\nn}{\nonumber \\}

 \newcommand{\res}{\mbox{\rm Res}}

\renewcommand{\hom}{\mbox{\rm Hom}}

\newcommand{\wt}{\mbox{\rm wt}\,}
\newcommand{\swt}{\mbox{\rm {\scriptsize wt}}\,}

\newcommand{\Y}{\mathcal{Y}}
\newcommand{\C}{\mathbb{C}}
\newcommand{\Z}{\mathbb{Z}}

\newcommand{\N}{\mathbb{N}}

\renewcommand{\l}{\llfloor}
\renewcommand{\r}{\rrfloor}

 \makeatletter
\newlength{\@pxlwd} \newlength{\@rulewd} \newlength{\@pxlht}
\catcode`.=\active \catcode`B=\active \catcode`:=\active \catcode`|=\active
\def\sprite#1(#2,#3)[#4,#5]{
   \edef\@sprbox{\expandafter\@cdr\string#1\@nil @box}
   \expandafter\newsavebox\csname\@sprbox\endcsname
   \edef#1{\expandafter\usebox\csname\@sprbox\endcsname}
   \expandafter\setbox\csname\@sprbox\endcsname =\hbox\bgroup
   \vbox\bgroup
  \catcode`.=\active\catcode`B=\active\catcode`:=\active\catcode`|=\active
      \@pxlwd=#4 \divide\@pxlwd by #3 \@rulewd=\@pxlwd
      \@pxlht=#5 \divide\@pxlht by #2
      \def .{\hskip \@pxlwd \ignorespaces}
      \def B{\@ifnextchar B{\advance\@rulewd by \@pxlwd}{\vrule
         height \@pxlht width \@rulewd depth 0 pt \@rulewd=\@pxlwd}}
      \def :{\hbox\bgroup\vrule height \@pxlht width 0pt depth
0pt\ignorespaces}
      \def |{\vrule height \@pxlht width 0pt depth 0pt\egroup
         \prevdepth= -1000 pt}
   }
\def\endsprite{\egroup\egroup}
\catcode`.=12 \catcode`B=11 \catcode`:=12 \catcode`|=12\relax
\makeatother

\makeatletter
\newcommand{\raisemath}[1]{\mathpalette{\raisem@th{#1}}}
\newcommand{\raisem@th}[3]{\raisebox{#1}{$#2#3$}}
\makeatother

\makeatletter%subalign
\newcommand{\subalign}[1]{%
	\vcenter{%
		\Let@ \restore@math@cr \default@tag
		\baselineskip\fontdimen10 \scriptfont\tw@
		\advance\baselineskip\fontdimen12 \scriptfont\tw@
		\lineskip\thr@@\fontdimen8 \scriptfont\thr@@
		\lineskiplimit\lineskip
		\ialign{\hfil$\m@th\scriptstyle##$&$\m@th\scriptstyle{}##$\hfil\crcr
			#1\crcr
		}%
	}%
}
\makeatother

\def\hboxtr{\FormOfHboxtr} % Only necessary if
\sprite{\FormOfHboxtr}(25,25)[0.5 em, 1.2 ex] % Resolution ca. 200x340 dpi.

:BBBBBBBBBBBBBBBBBBBBBBBBB |
:BB......................B |
:B.B.....................B |
:B..B....................B |
:B...B...................B |
:B....B..................B |
:B.....B.................B |
:B......B................B |
:B.......B...............B |
:B........B..............B |
:B.........B.............B |
:B..........B............B |
:B...........B...........B |
:B............B..........B |
:B.............B.........B |
:B..............B........B |
:B...............B.......B |
:B................B......B |
:B.................B.....B |
:B..................B....B |
:B...................B...B |
:B....................B..B |
:B.....................B.B |
:B......................BB |
:BBBBBBBBBBBBBBBBBBBBBBBBB |

\endsprite

\sprite{\FormOfShboxtr}(25,25)[0.3 em, 0.72 ex] % Resolution ca. 200x340 dpi.

:BBBBBBBBBBBBBBBBBBBBBBBBB |
:BB......................B |
:B.B.....................B |
:B..B....................B |
:B...B...................B |
:B....B..................B |
:B.....B.................B |
:B......B................B |
:B.......B...............B |
:B........B..............B |
:B.........B.............B |
:B..........B............B |
:B...........B...........B |
:B............B..........B |
:B.............B.........B |
:B..............B........B |
:B...............B.......B |
:B................B......B |
:B.................B.....B |
:B..................B....B |
:B...................B...B |
:B....................B..B |
:B.....................B.B |
:B......................BB |
:BBBBBBBBBBBBBBBBBBBBBBBBB |

\endsprite

\title{ {\bf  $C_{1}$-cofiniteness and vertex tensor categories} }
\date{}
\author{Yi-Zhi Huang}
\begin{document}

\bibliographystyle{alpha}
\maketitle
\begin{abstract}
We first generalize the logarithmic tensor category theory of 
Huang-Lepowsky-Zhang to the more general 
case that the module category for 
a vertex operator algebra $V$ (more generally a M\"{o}bius vertex algebra)
might not be closed under the 
contragredient functor. Then by verifying the assumptions
to use this generalization, we obtain that  (logarithmic) intertwining operators
among $C_{1}$-cofinite grading-restricted generalized $V$-modules 
satisfy the associativity property (operator product expansion)
and the category of $C_{1}$-cofinite grading-restricted generalized $V$-modules
has a natural vertex tensor category structure. 
In particular, this category has a  natural  
braided tensor category structure with a twist. 
\end{abstract}

\renewcommand{\theequation}{\thesection.\arabic{equation}}
\renewcommand{\thethm}{\thesection.\arabic{thm}}
\setcounter{equation}{0}
\setcounter{thm}{0}
\section{Introduction}

In the study of logarithmic conformal field theories and the 
nonsemisimple representation theory of vertex operator 
algebras, Lepowsky, Zhang and the author developed a 
logarithmic tensor category theory in \cite{HLZ0}--\cite{HLZ8}
by generalizing the early semisimple tensor category theory 
developed by Lepowsky and the author in  \cite{tensor0}--\cite{tensor3} and 
by the author in \cite{tensor4}. Under suitable assumptions
on a vertex operator algebra (more generally, a
M\"{o}bius vertex algebra) $V$ and 
a module category for $V$,   the associativity (operator product expansion)
of (logarithmic) intertwining operators among objects of the category 
is proved and a natural vertex tensor category 
structure and consequently a natural braided tensor category structure with a twist
are constructed in \cite{HLZ0}--\cite{HLZ8}. To use this theory, one needs 
to verify the assumptions in  \cite{HLZ0}--\cite{HLZ8}. 
See Assumption 10.1 in \cite{HLZ6} and 
Assumptions 12.1 and 12.2 in \cite{HLZ8} and also 
Section 2 below for a complete list of assumptions.

A notion of quasi-rational module for a $\mathcal{W}$-algebra 
was first introduced by Nahm in \cite{N}.  In mathematics, 
$\mathcal{W}$-algebras are formulated rigorously as vertex operator algebras and 
quasi-rational modules correspond to $C_{1}$-cofinite modules. 
For a suitable module $W$ with the vertex operator map $Y_{W}$ 
for a vertex operator algebra $V$,
we let $C_{1}(W)$ be the subspace spanned by 
$\res_{x}x^{-1}Y_{W}(v, x)w$ for $v\in V_{+}=\coprod_{n\in \Z_{+}}V_{(n)}$ and $w\in W$.
$W$ is $C_{1}$-cofinite means $\dim W/C_{1}(W)<\infty$.
Nahm in \cite{N}  argued 
that if a suitable fusion product $W_{12}$ 
of two $V$-modules $W_{1}$ and $W_{2}$ 
exists, then 
$$\dim ((W_{12})/C_{1}(W_{12}))
\le \dim (W_{1}/C_{1}(W_{1})) \dim (W_{2}/C_{1}(W_{2})).$$
Nahm did not give a construction of a fusion product in \cite{N}. 
Instead, Nahm's result is derived from some basic 
assumptions in the physical study of 
conformal field theory. 

In mathematics, Abe and Nagatomo showed in \cite{AN}
that for suitable vertex operator algebras, the spaces
of conformal blocks on the Riemann sphere among $C_{1}$-cofinite modules 
are finite-dimensional.
The importance of $C_{1}$-cofiniteness of suitable modules for 
vertex operator algebras in the study of the associativity of intertwining 
operators and in the construction of vertex and braided tensor category 
structures was first noticed by the author in \cite{H-diff-eqn}. 
 In \cite{H-diff-eqn}, the author proved that products of intertwining operators 
(without logarithms)
among $C_{1}$-cofinite modules for a vertex operator algebra
satisfy differential equations 
of regular singular points. In fact, the differential equations are also satisfied 
by products of logarithmic intertwining operators among 
$C_{1}$-cofinite grading-restricted generalized modules for a 
M\"{o}bius vertex algebra (see \cite{HLZ7}). Using these differential 
equations, the convergence and extension property for products of 
(logarithmic) intertwining operators are proved in \cite{H-diff-eqn}
and \cite{HLZ7}. It was also proved in \cite{HLZ7} that the expansion 
condition (one of the assumptions in \cite{HLZ0}--\cite{HLZ8}) follows from 
the convergence and extension property for products of 
(logarithmic) intertwining operators.

In \cite{H-finite-length}, the author 
verified the assumptions  in \cite{HLZ0}--\cite{HLZ8} 
for a $C_{2}$-cofinite vertex operator algebra
of positive energy (CFT type) and the category of 
grading-restricted generalized $V$-modules. 
In fact, the author in \cite{H-finite-length} 
verified these assumptions for a more general vertex operator algebra $V$
and the category of $C_{1}$-cofinite 
grading-restricted generalized $V$-modules satisfying additional conditions. 
But since the additional conditions are quite strong, 
the results in  \cite{H-finite-length} do not apply to many interesting
classes of vertex operator algebras and many 
interesting module categories. In \cite{M}, Miyamoto 
proved a weak version of the inequality of Nahm mentioned above and 
as a consequence, he obtained that the category of $C_{1}$-cofinite 
$N$-gradable weak $V$-modules
(or equivalently, $C_{1}$-cofinite grading-restricted generalized $V$-modules)
for a general vertex operator algebra $V$ is closed 
under a tensor product bifunctor. In \cite{MS}, McRae and Sopin formulated 
a generalization of the inequality of Nahm mentioned above (see Proposition 3.6 
in \cite{MS}) and observed that 
one can also obtain the inequality of Nahm using the number of 
independent solutions of the differential equations in \cite{M}.

 In \cite{CHY}, the results of \cite{H-diff-eqn} 
for $C_{1}$-cofinite $V$-modules 
was first applied to the case that $V$ is an affine vertex operator algebra at 
an admissible level and $\mathcal{C}$ is a semisimple category of 
$C_{1}$-cofinite ordinary $V$-modules. 
In \cite{CJORY}, 
Creutzig,  Jiang, Orosz Hunziker, Ridout and Yang
verified the assumptions in  \cite{HLZ0}--\cite{HLZ8} for 
a Virasoro vertex operator algebra $V$ and the category of 
$C_{1}$-cofinite grading-restricted generalized $V$-modules. 
In \cite{CY}, Creutzig and Yang 
showed that if all the irreducible $V$-modules are 
$C_1$-cofinite and the category of grading-restricted generalized
$V$-modules are the same as the category of finite-length grading-restricted 
generalized $V$-modules, then the assumptions 
in \cite{HLZ0}--\cite{HLZ8} hold. In fact, the following result was implicitly proved  in
 \cite{CJORY}  and \cite{CY}, explicitly stated and proved by McRae in \cite{Mc}
(see Theorem 2.3 in \cite{Mc}), and explicitly stated and generalized to 
vertex operator superalgebras by 
Creutzig, McRae, Orosz Hunziker  and Yang
in \cite{CMOY} (see Theorem 2.25 in \cite{CMOY}):

\begin{thm}[\cite{CJORY}, \cite{CY}, \cite{Mc}, \cite{CMOY}]\label{CMY}
Let $V$ be a vertex operator algebra. If the category of $C_{1}$-cofinite 
grading-restricted generalized $V$-modules is closed under the contragredient functor,
then it has a natural braided tensor category structure.
\end{thm}

Note that the results in \cite{CJORY}, \cite{CY}, \cite{Mc}, and \cite{CMOY} 
are obtained by applying the theory in \cite{HLZ0}--\cite{HLZ8}. 
So the theorem above is also true for a grading-restricted M\"{o}bius 
vertex algebra $V$. In addition, under the same condition, 
the associativity of (logarithnic) intertwining operator
among the category of $C_{1}$-cofinite 
grading-restricted generalized $V$-modules also holds and there is a 
natural vertex tensor category structure on this category. 

Theorem \ref{CMY} applies to 
$C_{1}$-cofinite module categories for many interesting vertex operator algebras, 
including, for example,  simple affine vertex operator algebras at all admissible 
and many non-admissible levels (\cite{CHY}, \cite{CY}), 
Virasoro vertex operator algebras of
all central charges (\cite{CJORY}),  the affine vertex operator superalgebra of 
$\mathfrak{gl}(1|1)$ (\cite{CMY3}), 
the singlet vertex operator algebras (\cite{CMY1}, \cite{CMY4}), 
the universal affine vertex operator algebra of $\mathfrak{sl}_2$ 
at admissible levels (\cite{MY}),
$N=1$ super Virasoro
vertex operator superalgebras of all central charges (\cite{CMOY}), and 
$N=2$ super Virasoro vertex operator algebras at central charge $\frac{3k}{k+2}$
(\cite{C}). 
Theorem  \ref{CMY} also applies indirectly to 
some non-$C_{1}$-cofinite module categories for some vertex operator algebras. 
In the case of the simple affine 
vertex operator algebra $V$ of $\mathfrak{sl}_2$ at an admissible level $k$, 
objects of the category of finitely-generated weight $V$-modules 
are in general not $C_{1}$-cofinite. But its Kazama-Suzuki dual (see \cite{KS}) 
is the $N=2$ super Virasoro vertex operator algebra of the central charges $\frac{3k}{k+2}$.
In \cite{C}, Creutzig used the vertex tensor category structure for 
this $N=2$ super Virasoro vertex operator algebra to obtain 
a vertex tensor category structure on the category of finitely-generated weight $V$-modules.

Despite the great progress discussed above, whether the category of 
$C_{1}$-cofinite grading-restricted modules 
for a general vertex operator algebra $V$ has a natural 
vertex tensor category structure is still an open problem. 
In this paper, we solve this problem. Here is the main theorem of the present paper:

\begin{thm}\label{main}
Let $V$ be a grading-restricted M\"{o}bius vertex algebra 
and $\mathcal{C}$ 
the category of $C_{1}$-cofinite grading-restricted generalized $V$-modules. 
Then  the associativity of (logarithmic) intertwining operators among objects of $\mathcal{C}$
holds and the category $\mathcal{C}$ has a natural vertex tensor category structure. 
In particular, the category $\mathcal{C}$ has a natural
braided tensor category structure with a twist. 
\end{thm}

This theorem is proved in Section 4. The proof is based on the generalizations and modifications 
given in this paper of the 
results and methods in \cite{N}, \cite{H-diff-eqn}, \cite{HLZ0}--\cite{HLZ8}, \cite{M}, 
\cite{CJORY}, \cite{CY}, and \cite{Mc}. 
We also note that Theorem \ref{main} above and 
Theorem 2.25 in \cite{CMOY} can be generalized easily to 
a corresponding result for a grading-restricted 
M\"{o}bius vertex superalgebra $V$.

The main reason why Theorem \ref{CMY} 
still needs the condition that the category is closed under the contragredient functor
is because  this condition is one of the assumptions 
in \cite{HLZ0}--\cite{HLZ8}. We cannot apply directly the logarithmic 
tensor category theory \cite{HLZ0}--\cite{HLZ8} to the  
category of $C_{1}$-cofinite grading-restricted modules  since 
in general, this category might not be closed under the contragredient bifunctor. 
In this paper, we generalize the theory 
in \cite{HLZ0}--\cite{HLZ8} by removing this assumption.  
We obtain several assumptions in this generalization such that 
when these assumptions are satisfied, the category that we are considering 
has natural vertex and braided tensor category structures. 

To verify these assumptions in this generalization, 
we show that the arguments and proofs in \cite{H-diff-eqn} used 
to derive the differential equations 
still  works in the case that 
the contragredient of a $C_{1}$-cofinite grading-restricted $V$-module
might not be $C_{1}$-cofinite. We also recall some results of 
Miyamoto in \cite{M} and give a proof of the generalization formulated in \cite{MS}
of the inequality of Nahm 
 in \cite{N} without using the differential equations in \cite{M}. 
Using these results, we  verify that for any grading-restricted M\"{o}bius vertex algebra
(in particular,
a vertex operator algebra)  $V$ and 
 the category of $C_{1}$-cofinite grading-restricted $V$-modules,
 the assumptions to use our generalization of the theory in \cite{HLZ0}--\cite{HLZ8}
are satisfied. In fact, the proofs of the relevant results in \cite{CJORY}, \cite{CY},  \cite{Mc}
and \cite{CMOY} 
can  be modified to verify some of these assumptions. 
For the reader's convenience,  we still give the full details of the verification. 
Then by our generalization, we obtain the associativity (operator product expansion)
of (logarithmic) intertwining operators among
$C_{1}$-cofinite grading-restricted $V$-modules and 
natural vertex and braided tensor category structures
on this category.  

Theorem \ref{main} above
confirms the conjecture that 
$C_{1}$-cofiniteness is a sufficient condition for the existence of genus-zero chiral (logarithmic) 
conformal field theories. 
A consequence of the $C_{1}$-cofiniteness condition is that the genus-zero conformal blocks
are finite-dimensional since spaces of solutions of differential equations 
in finite-dimensional spaces are finite-dimensional. To go beyond 
conformal field theories with finite-dimensional genus-zero conformal blocks,
it is necessary to study vertex operator algebras and modules that do not 
satisfy the $C_{1}$-confiniteness condition. The Liouville conformal field theory 
is an important example  of such non-$C_{1}$-cofinite conformal field theories.
The mathematical construction of the Liouville conformal field theory by  
Kupiainen, Rhodes and Vargas in \cite{KRV} and by 
Guillarmou, Kupiainen, Rhodes and Vargas in \cite{GKRV1} and \cite{GKRV2}
uses the probability approach, which is heavily based 
on analysis in infinite-dimensional spaces. The chiral 
part of the Liouville conformal field theory can be used to study 
the representation theory of Virasoro vertex operator algebras and their modules. 
The author conjectures that we might be able to reformulate some part of 
the infinite-dimensional analysis 
used in the works \cite{KRV},
 \cite{GKRV1} and \cite{GKRV2} as a theory of 
differential equations in infinite-dimensional spaces. If such a theory  can be developed
in the future, non-$C_{1}$-cofinite modules for vertex operator algebras can also be studied 
using the representation theory of vertex operator algebras together with 
this to-be-developed theory of  differential equations in infinite-dimensional spaces. 

This paper is organized as follows: In the next section,
we generalize the logarithmic tensor category in \cite{HLZ0}--\cite{HLZ8}.
The assumptions for applying this generalization are given in this section.
In Section 3, we show that the arguments and proofs used to derive 
 in \cite{H-diff-eqn} the differential equations 
still works in the case that the generalized $V$-modules placed at $0$ and $\infty$
are quasi-finite-dimensional but not $C_{1}$-cofinite. 
We also recall some results of 
Miyamoto on $C_{1}$-cofinite $\N$-gradable $V$-modules 
and give a proof of a generalization of the inequality of Nahm in this section.
We verify the assumptions of our generalization in Section 2 for
a grading-restricted M\"{o}bius vertex algebra $V$ and 
 the category of $C_{1}$-cofinite grading-restricted $V$-modules 
in Section 4. In particular, we obtain the main result Theorem \ref{main} of the present paper. 

\paragraph{Acknowledgment} 
The author is grateful to Thomas Creutzig,  Robert McRae and Jinwei Yang
for helpful comments and discussions.  I am especially grateful to Robert McRae
for discussions on the inequality of Nahm.

\renewcommand{\theequation}{\thesection.\arabic{equation}}
\renewcommand{\thethm}{\thesection.\arabic{thm}}
\setcounter{equation}{0}
\setcounter{thm}{0}
\section{A generalization of the logarithmic tensor category}

In this section, we first review briefly the assumptions and 
results in the logarithmic tensor category theory 
developed in \cite{HLZ0}--\cite{HLZ8}. The construction and results 
in \cite{HLZ0}--\cite{HLZ8} work for a suitable 
vertex algebra $V$ much more general than 
a vertex operator algebra. In general, $V$ can be a 
strongly $A$-graded M\"{o}bius vertex algebra (see \cite{FHL} and \cite{HLZ1}
for the precise definition of M\"{o}bius vertex algebra), where $A$
is an abelian group. 
In this paper, we are interested only in the case that $A$ is trivial
and $V$ is a grading-restricted M\"{o}bius vertex algebra.  
So for simplicity, we do not discuss the general case that $A$ is not trivial. 
But it is clear that  the generalization 
obtained in this section still works in the general case. 

One of the main assumption in \cite{HLZ0}--\cite{HLZ8} is 
that the module category for $V$
should be closed under the contragredient functor. 
In this section, we show that the construction and results in 
\cite{HLZ0}--\cite{HLZ8} can be generalized to the case that the 
module category might not be closed under 
the contragredient functor. 

In this paper, we fix a  grading-restricted 
M\"{o}bius vertex algebra $V$. 
(Using the terminology in \cite{HLZ1}, $V$ is a  M\"{o}bius vertex algebra
strongly graded with respect to the weight-grading and 
the trivial abelian group grading.) So when we recall the material in 
\cite{HLZ0}--\cite{HLZ8}, the abelian groups $A$ and $\tilde{A}$ 
giving the horizontal gradings of $V$ and its modules, respectively,
are all trivial. 

Though we are mainly interested in grading-restricted generalized $V$-modules,
other types of $V$-modules always appear in our proofs and intermediate steps. 
This is an important feature of the representation theory of 
vertex  (operator) algebras. So we first briefly discuss our terminology 
for these different types of $V$-modules and their differences. For the precise definitions, see 
\cite{HLZ1} and \cite{H-finite-length}. 

A generalized $V$-module is a $\C$-graded vector 
space $W$ equipped with a vertex operator map
$Y_{W}: V\otimes W\to W[[x, x^{-1}]]$ and operators 
$L_{W}(-1)$, $L_{W}(0)$, $L_{W}(1)$ satisfying all the axioms
for a M\"{o}bius vertex algebra that still make sense.
Note that though $W$ has a $\C$-grading, it does not have to 
be lower bounded or grading restricted. 
A lower-bounded  generalized $V$-module is a generalized 
$V$-module $W=\coprod_{n\in \C}W_{[n]}$ such that
$W_{[n]}=0$ for $\Re(n)<N$ for some $N\in \Z$. 
A grading-restricted generalized $V$-module is 
a lower-bounded generalized $V$-module 
$W=\coprod_{n\in \C}W_{[n]}$ such that 
$\dim W_{[n]}<\infty$ for $n\in \C$. An ordinary $V$-module is a
grading-restricted generalized $V$-module such that $L(0)$ 
acts on the generalized $V$-module semisimply.
A quasi-finite-dimensional 
generalized $V$-module is a generalized $V$-module 
 $W=\coprod_{n\in \C}W_{[n]}$ such that 
$\dim \coprod_{\Re(n)<N}W_{[n]}<\infty$ for $N\in \Z$.

A weak $V$-module is a vector space $W$ and a vertex operator 
map $Y_{W}: V\otimes W\to W[[x, x^{-1}]]$ satisfying all the axioms,
including the Jacobi identity, for generalized $V$-modules except for those
involving the grading of $W$.  For a weak $V$-module $W$, 
an $\N$-grading $W=\coprod_{n\in \N}W_{\l n\r}$ on $W$ is said to be 
compatible if for $v\in V_{(m)}$ and $W_{\l n\r}$, 
$v_{k}w\in W_{\l m-k-1+n\r}$,
where $v_{n}=\res_{x}x^{n}Y_{W}(v, x)$. An 
$\N$-gradable weak $V$-module is a weak $V$-module 
for which there exists a compatible $\N$-grading. 

Intertwining operators among any types of $V$-modules above 
are well defined. Since $L(0)$ might act nonsemisimply on generalized $V$-modules,
they might contain logarithms of the variables. 
These are called logarithmic intertwining operators. For simplicity,
we shall simply call them intertwining operators also unless it is necessary to 
emphasize there are logarithms of the variables. 

Let $\mathcal{C}$ be a full subcategory of 
generalized $V$-modules.
The construction of the vertex tensor category and braided tensor category structures
in \cite{HLZ0}--\cite{HLZ8} are obtained based the following 
assumptions on $\mathcal{C}$
 (see Assumptions 10.1 in \cite{HLZ6}, Assumptions 12.1 and
12.2  in \cite{HLZ8}): 

\begin{enumerate}

\item Objects of $\mathcal{C}$ are grading-restricted. 

\item For any object $W$ in $\mathcal{C}$, the weights of homogeneous 
elements are real numbers and there exists $K\in \Z_{+}$ such that 
$L_{W}(0)_{N}^{K}=0$ where $L_{W}(0)_{N}$ is the nilpotent part of 
$L_{W}(0)$. 

\item $\mathcal{C}$ is closed under images, 
under the contragredient functor, under taking finite direct
sums, and under $P (z)$-tensor products $\boxtimes_{P(z)}$ 
for some $z\in \C^{\times}$.

\item $V$ is an object of $\mathcal{C}$.

\item  The convergence and expansion conditions
for intertwining maps in $\mathcal{C}$ hold.

\item The products of more than two
intertwining operators are absolutely convergent in the corresponding region
and can be analytically extended to multivalued analytic 
functions defined on the region where the complex variables 
in the intertwining operators 
are not equal $0$ and each other.

\end{enumerate}
It is proved in \cite{HLZ0}--\cite{HLZ8} that if 
Assumptions 
1--6 hold, then intertwining operators among objects of $\mathcal{C}$ 
satisfy the associativity property
and $\mathcal{C}$ has a natural vertex tensor category 
structure.
In particular, $\mathcal{C}$ has a natural braided tensor category structure with a twist.
 
Theorem 11.4 in \cite{HLZ7} states that Assumption 5 above 
holds if the following Assumptions
hold:

\begin{enumerate}
\setcounter{enumi}{6}

\item Every finitely-generated lower-bounded 
generalized $V$-module is an object
of $\mathcal{C}$. 

\item The convergence and extension property for either products or iterates 
holds in $\mathcal{C}$.

\end{enumerate}

By the properties of intertwining operators and 
Proposition 2.1 In \cite{H-applicability}, the first part of Assumption 2 above, that
the weights of homogeneous 
elements are real numbers, can be replaced by the assumption that 
for any object $W$ of $\mathcal{C}$, the weights of homogeneous 
elements of $W$ are congruent to finitely many complex numbers modulo $\Z$. 
So Assumption 2 can be replaced by the following weaker assumption:

\begin{enumerate}
\setcounter{enumi}{8}

\item For any object $W$ of $\mathcal{C}$, the weights of homogeneous 
elements of $W$ are congruent to finitely many complex numbers modulo $\Z$
and there exists $K\in \Z_{+}$ such that 
$L_{W}(0)_{N}^{K}=0$ where $L_{W}(0)_{N}$ is the nilpotent part of 
$L_{W}(0)$. 

\end{enumerate}

\begin{rema}
{\rm If an object of $\mathcal{C}$ satisfies Assumption 9, then certainly 
the contragredient of this object also satisfies Assumption 9. This fact is needed
in our generalization of \cite{HLZ0}--\cite{HLZ8}. For example, 
(9.139) in \cite{HLZ6} is proved using the fact that $W_{1}\hboxtr_{P(z)}W_{2}$ 
is an object of $\mathcal{C}$. In our generalization, $\mathcal{C}$ is closed under 
$P(z)$-tensor product, $W_{1}\hboxtr_{P(z)}W_{2}$ as the contragredient 
of the $P(z)$-tensor product is in general not an object of $\mathcal{C}$. 
But since its contragredient is an object of $\mathcal{C}$, $W_{1}\hboxtr_{P(z)}W_{2}$
still satisfies  Assumption 9.}
\end{rema}

Also in \cite{H-applicability}, it is proved that Assumption 7 above 
can be replaced by the following weaker assumption:

\begin{enumerate}
\setcounter{enumi}{9}

\item For any objects $W_1$ and $W_2$ of $\mathcal{C}$ 
and any $z \in \C^{\times}$, if the generalized $V$-module 
$W_{\lambda}$ generated by a generalized
eigenvector $\lambda\in (W_1 \otimes W_2)^{*}$ for 
$L_{P (z)}(0)$ satisfying the $P (z)$-compatibility condition is
lower bounded, then $W_{\lambda}$ is an object of $\mathcal{C}$.

\end{enumerate}

By  these results on the assumptions, we 
see that for a M\"{o}bius vertex algebra $V$ and  
a full subcategory of 
generalized $V$-modules $\mathcal{C}$ satisfying 
Assumptions 1, 3, 4, 6, 8, 9, and10,  associativity of intertwining operators 
among objects of $\mathcal{C}$ is satisfied and $\mathcal{C}$ has a natural 
vertex tensor category structure and thus a braided tensor 
category structure with a twist. 
But this result still cannot be applied to the case that 
$\mathcal{C}$ is the category 
of $C_{1}$-cofinite 
grading-restricted generalized $V$-modules because 
in general, the contragredient 
of a $C_{1}$-cofinite 
grading-restricted generalized $V$-module might not be 
$C_{1}$-cofinite.  Now we generalize the construction 
in \cite{HLZ0}--\cite{HLZ8} to the case that $\mathcal{C}$ 
might not be closed under the contragredient functor. 

We need to replace Assumption 3 by the following 
weaker assumption:

\begin{enumerate}
\setcounter{enumi}{10}

\item $\mathcal{C}$ is closed under images,  under taking finite direct
sums, and under $P (z)$-tensor products $\boxtimes_{P(z)}$ 
for some $z\in \C^{\times}$.

\end{enumerate}

Let $\mathcal{C}$ be a full subcategory of 
generalized $V$-modules satisfying 
Assumptions 1 and closed under images and finite direct sums. 
Even though  $\mathcal{C}$ might not be closed 
under the contragredient functor,
contragredients of objects of $\mathcal{C}$ are 
grading-restricted generalized $V$-modules. 
 
In this case, the definition of 
$W_{1}\hboxtr_{P(z)}W_{2}$ in the category $\mathcal{C}$
is in fact the same as the one in \cite{HLZ4} (see Definition 5.31). For the 
reader's convenience, we recall Definition 5.31 in \cite{HLZ4} here.

\begin{defn}\label{hboxtr}
{\rm For objects $W_{1}, W_{2}$ in $\mathcal{C}$, define the subset 
$$W_{1}\hboxtr_{P(z)}W_{2}\subset (W_{1}\otimes W_{2})^{*}$$
of $(W_{1}\otimes W_{2})^{*}$ to be the union of the images
$$I'(W')\subset (W_{1}\otimes W_{2})^{*}$$
as $(W; I)$ ranges through all the $P(z)$-products of $W_{1}$ and
$W_{2}$ with objects $W$ in $\mathcal{C}$. Equivalently,
$W_{1}\hboxtr_{P(z)}W_{2}$ is the union of the images $I'(W')$ as $W$
ranges through objects in $\mathcal{C}$ and $I'$ ranges through
the space of linear maps
$$W'\to (W_{1}\otimes W_{2})^{*}$$
intertwining the actions of 
$$V\otimes \iota_{+}\C[t, t^{-1}, (z^{-1}-t)^{-1}]$$
and $L(-1)$, $L(0)$ and $L(1)$ on both spaces.}
\end{defn}

\begin{rema}
{\rm Note that in Definition 5.31 in \cite{HLZ4},
we have the phrase ``as $W$
(or $W'$) ranges through $\mathrm{ob}\,\mathcal{C}$"
 in the second 
equivalent formulation of the definition. But in the definition above
we have deleted  ``(or $W'$)" because in this case, $W'$
might not be  an object of $\mathcal{C}$. }
\end{rema}

The first result in \cite{HLZ4} that needs to be generalized is 
Proposition 5.37. We also need to generalize  Proposition 5.36
in \cite{HLZ4}. But we need to use the generalization of 
Proposition 5.37 in \cite{HLZ4} to generalize
Proposition 5.36 in \cite{HLZ4}.

\begin{prop}\label{HLZ4-Prop5.37}
Assume that $\mathcal{C}$ satisfies Assumptions 1 and 
is closed under images and finite direct sums. 
Let $W_{1}, W_{2}$ be objects of $\mathcal{C}$.  If the contragredient 
of $W_1\hboxtr_{P(z)} W_2$ is an object of $\mathcal{C}$, then
the $P(z)$-tensor product of $W_{1}$ and $W_{2}$ in $\mathcal{C}$
exists and is
$(W_1\boxtimes_{P(z)} W_2,  i')$,
where $W_1\boxtimes_{P(z)} W_2 = (W_1\hboxtr_{P(z)} W_2)'$ and 
$i$ is
the natural inclusion {}from $W_1\hboxtr_{P(z)} W_2$ to $(W_1\otimes
W_2)^*$.  Conversely, if the $P(z)$-tensor
product of $W_1$ and $W_2$ in $\mathcal{C}$ exists, then the contragredient of 
$W_1\hboxtr_{P(z)} W_2$ is an object of $\mathcal{C}$.
\end{prop}
\begin{proof}
The proof of the first part is the same as in \cite{HLZ4}. 
For the converse, without using contragredients 
of objects of $\mathcal{C}$, the proof of Proposition 5.37 in \cite{HLZ4} 
shows that $W_{1}\hboxtr_{P(z)}W_{2}=I_{0}'(W_{0}')$, where 
$(W_{0}, I_{0})$ is the $P(z)$-tensor product of $W_{1}$ and $W_{2}$. 
By Proposition 4.23 in \cite{HLZ3} (which is a straightforward generalization of 
Lemma 4.9 in \cite{tensor4}),  we know that the homogeneous components of 
elements of $\overline{W}_{0}$ of the form $I_{0}(w_{1}\otimes w_{2})$
for $w_{1}\in W_{1}$ and $w_{2}\in W_{2}$ span $W_{0}$. Then $I_{0}'$ is injective. 
This means that as a $V$-module map from $W_{0}'$ to $I'(W_{0}')$, $I_{0}'$
is an equivalence. Hence its adjoint from the contragredient of $I_{0}'(W_{0}')$
to the contragredient $W_{0}''$ of $W_{0}'$ is also an equivalence. 
Since objects of $\mathcal{C}$ are grading-restricted, 
$W_{0}''$ is equivalent to $W_{0}$, which is an object of $\mathcal{C}$. 
Thus $W_{0}''$ and  the contragredient of $I'(W_{0}')$ are also 
objects of $\mathcal{C}$ since $\mathcal{C}$ is closed under images. 
Since $W_{1}\hboxtr_{P(z)}W_{2}=I_{0}'(W_{0}')$, we see that 
the contragredient of 
$W_1\hboxtr_{P(z)} W_2$ is  also an object of $\mathcal{C}$.
\end{proof}

We now assume that $\mathcal{C}$ satisfies Assumptions 1 and 11. 
In this case, since the $P(z)$-tensor
product of $W_1$ and $W_2$ in $\mathcal{C}$ exists,
by Proposition \ref{HLZ4-Prop5.37},  the contragredient 
$(W_1\hboxtr_{P(z)}W_2)'$ of  
$W_1\hboxtr_{P(z)}W_2$
is an object of $\mathcal{C}$. Then we let 
$W_{1}\boxtimes_{P(z)}W_{2}=(W_1\hboxtr_{P(z)}W_2)'$
and take the $P(z)$-tensor product
of $W_{1}$ and $W_{2}$ to be $(W_{1}\boxtimes_{P(z)}W_{2}, i')$,
where $i: W_1\hboxtr_{P(z)}W_2\to (W_{1}\otimes W_{2})^*$
is the inclusion map and $i': W_{1}\otimes W_{2}\to W_{1}\boxtimes_{P(z)}W_{2}$
is the $P(z)$-intertwining map corresponding to $i$. 

We are ready to generalize 
Proposition 5.36 \cite{HLZ4} now. The proof of  Proposition 5.36 in \cite{HLZ4} 
uses the assumption that the category there is closed under 
contragredient functor. Since $\mathcal{C}$ might not be closed under 
the contragredient functor, we need to use the converse part 
of Proposition \ref{HLZ4-Prop5.37},
which in  turn uses the assumption that 
$\mathcal{C}$ is closed under 
$P (z)$-tensor products $\boxtimes_{P(z)}$ 
for some $z\in \C^{\times}$.

\begin{prop}\label{backslash=union}
Assume that $\mathcal{C}$ satisfies Assumptions 1 and  11. 
Let $W_1,W_2$ be objects of $\mathcal{C}$.  Then the subspace 
$W_1\hboxtr_{P(z)}
W_2$ of $(W_{1}\otimes W_{2})^{*}$ is equal to the union and also to
the sum of grading-restricted generalized $V$-modules contained in 
$(W_{1}\otimes
W_{2})^{*}$ (equipped with the action $Y'_{P(z)}(\cdot,x)$ of $V$ and the
corresponding actions of $L(-1)$, $L(0)$ and $L(1)$ on
$(W_{1}\otimes W_{2})^{*}$)
 such that their contragredients are objects of $\mathcal{C}$.
\end{prop}
\begin{proof}
The only difference between  Proposition 5.36 in \cite{HLZ4} 
and this proposition is that here the sum is over grading-restricted 
generalized $V$-modules which might not be objects of $\mathcal{C}$,
but whose contragredients are objects of $\mathcal{C}$. 
Since by Assumption 1, an object $W$ of $\mathcal{C}$
is grading-restricted,  $W'$ is also grading-restricted. 
Then for an object $W$ of $\mathcal{C}$ and  a $P(z)$-intertwining map $I$
of type $\binom{W}{W_{1}W_{2}}$, $I'(W')$ is also grading restricted. 
So the only thing we need to prove is that the contragredient of 
$I'(W')$ is an object of $\mathcal{C}$. 

Since $I'(W')$ is a generalized $V$-submodule of $W_{1}\hboxtr_{P(z)}$,
we have the inclusion map $i$ from $I'(W')$ to $W_{1}\hboxtr_{P(z)}W_{2}$.
The inclusion map $i$ is an injective $V$-module map. By Proposition \ref{HLZ4-Prop5.37},
$W_{1}\hboxtr_{P(z)}W_{2}$ is grading-restricted. Then the adjoint $i'$ 
of $i$ is a surjective $V$-module map from 
$W_{1}\boxtimes_{P(z)}W_{2}= (W_{1}\hboxtr_{P(z)}W_{2})'$
to the contragredient of 
$I'(W')$. Since $\mathcal{C}$ is closed under images, 
the contragredient of 
$I'(W')$ is also an object of $\mathcal{C}$.
\end{proof}

In the work \cite{HLZ4}, one crucial result is 
that elements of $W_{1}\hboxtr_{P(z)}W_{2}$ can be characterized by 
two conditions called the $P(z)$-compatibility condition and 
$P(z)$-local grading restriction condition. See Section 5 in \cite{HLZ4} for 
these conditions and Theorem 5.50 in 
\cite{HLZ4} for the theorem. But the condition of Theorem 
5.50 in \cite{HLZ4} is not satisfied for our category $\mathcal{C}$, which is not 
closed under the contragredient functor. In fact, 
the condition of Theorem 
5.50 in \cite{HLZ4} states that every element of $(W_{1}\otimes W_{2})^{*}$
satisfies the $P(z)$-compatibility condition and the 
$P(z)$-local grading-restriction condition must contain in
a generalized $V$-submodule of some object of $\mathcal{C}$.
But for an object $W$ of $\mathcal{C}$ and a $P(z)$-intertwining map 
$I$ of type $\binom{W}{W_{1}W_{2}}$, we obtain an 
element $I'(w')\in (W_{1}\otimes W_{2})^{*}$ for each $w'\in W'$ satisfying 
both the $P(z)$-compatibility condition and the 
$P(z)$-local grading-restriction condition. But since $W'$ might not be an 
object of $\mathcal{C}$, $I'(w')$ also might not 
be in some object of $\mathcal{C}$. So we need to generalize
Theorem 5.50 in \cite{HLZ4}.

\begin{thm}\label{characterization}
Suppose that for every element 
$\lambda
\in (W_1\otimes W_2)^*$ satisfying both the 
$P(z)$-compatibility condition and the $P(z)$-local grading
restriction condition,
the generalized module $W_{\lambda}$ generated by $\lambda$ 
as given in Theorem 5.49 in \cite{HLZ4}
 is a generalized $V$-submodule of a grading-restricted generalized 
$V$-module in $(W_{1}\otimes W_{2})^{*}$ whose 
contragredient is an object of
$\mathcal{C}$. Then
$W_1\hboxtr_{P(z)}W_2$ is equal to the space of all such $\lambda$. 
\end{thm}
\begin{proof}
We denote the space of all such $\lambda$ by $W_{3}$. 
We need to prove that $W_1\hboxtr_{P(z)}W_2=W_{3}$.
We prove $W_1\hboxtr_{P(z)}W_2\subset W_{3}$ and 
$W_{3}\subset W_1\hboxtr_{P(z)}W_2$. 

From the calculations and discussions
before the statements of $P(z)$-compatibility condition and 
the $P(z)$-local grading
restriction condition in
Section 5 of \cite{HLZ4}, we know that 
for an object $W$ of $\mathcal{C}$,
a $P(z)$-intertwining map of type $\binom{W}{W_{1}W_{2}}$ 
and an element $w'\in W'$, the element $I'(w')$
of $(W_{1}\otimes W_{2})^{*}$ satisfies 
these two conditions. Also, the generalized $V$-module 
$W_{I(w')}$ is a generalized $V$-submodule of the 
grading-restricted generalized $V$-module $I'(W')$.
By Proposition \ref{backslash=union}, the contragredient of 
$I'(W')$ is an object of $\mathcal{C}$. We have proved  $I(w')\in W_{3}$.
Thus we have $W_1\hboxtr_{P(z)}W_2\subset W_{3}$.

Let $\lambda\in W_{3}$. By Theorem 5.49 in \cite{HLZ4}, 
$W_{\lambda}\subset (W_{1}\otimes W_{2})^{*}$ is a
grading-restricted generalized $V$-module. By assumption, 
$W_{\lambda}$ is the generalized $V$-submodule  of 
a grading-restricted generalized 
$V$-module $W$ in $(W_{1}\otimes W_{2})^{*}$ whose 
contragredient is an object of
$\mathcal{C}$. In particular, $\lambda\in W$. Since 
the contragredient of $W$ is an object of $\mathcal{C}$, 
by Proposition 
\ref{backslash=union}, $W\subset W_1\hboxtr_{P(z)}W_2$.
In particular, $\lambda\in W_1\hboxtr_{P(z)}W_2$. 
Thus we have $W_{3}\subset W_1\hboxtr_{P(z)}W_2$. 
\end{proof}

Another important result that we need to generalize is Theorem 3.1 
in \cite{H-applicability}, which in turn is a generalization of 
Theorem 11.4 in \cite{HLZ7}. To generalize this result, we first need to 
reformulate Condition 1 in this result, that is, Assumption 10
above. The following assumption is what we use to replace 
Assumption 10:

\begin{enumerate}
\setcounter{enumi}{11}

\item For any objects $W_1$ and $W_2$ of $\mathcal{C}$ 
and any $z \in \C^{\times}$, if the generalized $V$-module 
$W_{\lambda}$ generated by a generalized
eigenvector $\lambda\in (W_1 \otimes W_2)^{*}$ for 
$L_{P (z)}'(0)$ satisfying the $P (z)$-compatibility condition is
lower bounded, then $W_{\lambda}$ is grading-restricted and its 
contragredient $W_{\lambda}'$ is an object of $\mathcal{C}$.
\end{enumerate}

\begin{thm}\label{expansion-cond}
Suppose that in addition to Assumptions 1 and  11, Assumption 
8, 9 and 12 also hold. 
Then the convergence and expansion conditions for intertwining maps in
$\mathcal{C}$ both hold.
\end{thm}
\begin{proof}
This proof is based on the proof of 
Theorem 11.4 in \cite{HLZ7}. The reader should read 
that proof there to see what is modified and generalized here. 

Note that the convergence condition 
for intertwining maps in the category $\mathcal{C}$
holds since Assumption 8 in particular gives this condition. 

As mentioned in the proof of Theorem 3.1 
in \cite{H-applicability}, the condition in Theorem 11.4 in \cite{HLZ7}
that every finitely-generated lower 
bounded generalized $V$-module is in $\mathcal{C}$ 
(Condition 1 in that theorem)  is only used  in the last paragraph 
showing that $W_{\lambda_{n}^{(2)}(w'_{(4)}, w_{(3)})}$ is an object of  $\mathcal{C}$.
In our case, $W_{\lambda_{n}^{(2)}(w'_{(4)}, w_{(3)})}$ might not 
be an object of $\mathcal{C}$. Instead, we need to show that the contragredient 
of $W_{\lambda_{n}^{(2)}(w'_{(4)}, w_{(3)})}$ is an object of $\mathcal{C}$. 
But it is proved in the proof of Theorem 11.4 in \cite{HLZ7}
that $\lambda_{n}^{(2)}(w'_{(4)}, w_{(3)})$
is a generalized eigenvector of $L'_{P(z)}(0)$ with eigenvalue $n$ and,
by using Theorem 9.17 in \cite{HLZ6}, that $\lambda_{n}^{(2)}(w'_{(4)}, w_{(3)})$
satisfies the $P(z)$-compatibility condition. It is also shown in the 
proof of  Theorem 11.4 in \cite{HLZ7} that 
$W_{\lambda_{n}^{(2)}(w'_{(4)}, w_{(3)})}$ is lower-bounded. 
Then by Assumption 12, $W_{\lambda_{n}^{(2)}(w'_{(4)}, w_{(3)})}$
is grading-restricted and its 
contragredient  is an object of $\mathcal{C}$. But an object of 
$\mathcal{C}$ is grading-restricted and its contragredient is also grading-restricted. 
So $W_{\lambda_{n}^{(2)}(w'_{(4)}, w_{(3)})}$ is grading-restricted. 
This shows that $I_{1}\circ
(1_{W_{2}}\otimes I_{2}))'(w'_{(4)})$ of $(W_{1}\otimes W_{2}\otimes
W_{3})^{*}$ satisfies the $P^{(2)}(z_{1}-z_{2})$-local
grading-restriction condition. Thus the expansion condition
for intertwining maps in the category $\mathcal{C}$ holds. 
\end{proof}

Theorem \ref{expansion-cond} in fact says that Assumptions 1, 
8, 9, 11, and 12 implies Assumption 5.

Besides the generalizations above of the corresponding results in 
\cite{HLZ5}, we also need to modify all the statements in \cite{HLZ4}--\cite{HLZ8} involving 
statements on the contragredients of objects of $\mathcal{C}$, especially the 
statements on $W_{1}\hboxtr_{P(z)}W_{2}$ and on generalized $V$-modules
in $(W_{1}\otimes W_{2})^{*}$.
Even under the assumption that the $P(z)$-tensor product of objects $W_{1}$
and $W_{2}$ of $\mathcal{C}$ exist, $W_{1}\hboxtr_{P(z)}W_{2}$
and its generalized $V$-submodules 
are in general not objects of $\mathcal{C}$. So when we see 
a statement saying that $W_{1}\hboxtr_{P(z)}W_{2}$ or a generalized $V$-module 
in $(W_{1}\otimes W_{2})^{*}$ is an object of $\mathcal{C}$, we need to change 
the statement to say that $W_{1}\hboxtr_{P(z)}W_{2}$ or a generalized $V$-module 
in $(W_{1}\otimes W_{2})^{*}$ is a grading-restricted generalized $V$-module
whose contragredient is an object of $\mathcal{C}$. The proofs are completely the same except that 
we use the generalizations above instead of the corresponding results in \cite{HLZ4}.
In addition to the generalizations above of the results related to $P(z)$-tensor products
in \cite{HLZ4}--\cite{HLZ8},
we also need to generalize of the results on $Q(z)$-tensor products in \cite{HLZ4}. 
But since these generalizations are completely the same as those we have given above
for $P(z)$-tensor products,
we omit the detailed discussions here. 

More specifically, the proofs of the main results in \cite{HLZ6} use Propositions 5.36, 5.37 
and Theorem  5.50 in \cite{HLZ4}. Since we have generalized these results to 
Propositions \ref{backslash=union}, \ref{HLZ4-Prop5.37} and 
Theorem \ref{characterization}, the corresponding statements 
of some main results in \cite{HLZ6}  also need to be generalized as discussed above.
As mentioned above, 
the proofs of these generalizations are completely the same as the 
proofs of the corresponding results in \cite{HLZ6} except for the 
corresponding statements about the
objects of $\mathcal{C}$.  So here we do not give the 
full statements and the proofs of these generalizations.  
In fact, only one phrase needs to be changed in all these results. 
Here is a list of results in \cite{HLZ6} and the change of one phrase needed 
to generalize these results on this list: 
In the statements of Propositions 9.13, Remark 9.20, Corollary 9.21, 
Theorem 9.23, Corollary 9.24, Theorem 9.27, and the proof of Theorem 10.3 in
\cite{HLZ6}, replace the phrase ``a generalized
$V$-submodule (or a $V$-submodule) of some object of $\mathcal{C}$" by
``a generalized $V$-submodule (or a $V$-submodule) of the contragredient of 
some object of $\mathcal{C}$."

Using the generalizations above, we see that the main results 
in  \cite{HLZ6} and \cite{HLZ8} still hold. In summary,  we obtain 
the following main theorem of this section:

\begin{thm}\label{1478910-vtc}
Let $V$ be a  M\"{o}bius vertex algebra and $\mathcal{C}$ 
a full subcategory of 
generalized $V$-modules. Assume that Assumptions 
1, 4,  6, 8, 9, 11, and 12 above hold. 
Then the associativity of intertwining operators among objects of $\mathcal{C}$
holds and
$\mathcal{C}$ has a natural vertex tensor category structure.
In particular, $\mathcal{C}$ has a natural  braided tensor category structure with a twist.
\end{thm}

\begin{rema}
{\rm For the definition 
of vertex tensor category based on spheres with punctures and 
local coordinates vanishing at punctures,
see \cite{HL}. If $V$ is a vertex operator algebra, a vertex tensor category structure 
is essentially constructed in \cite{HLZ8} but is not 
stated explicitly. The tensor product bifunctors associated to 
conformal classes of general spheres with three punctures and local coordinates
and the corresponding associated  isomorphisms, commutativity isomorphisms
and so on can be obtained using the $P(z)$-tensor product bifunctors 
and the corresponding associativity isomorphisms,
commutativity isomorphisms and so on given in \cite{HLZ8} and 
the work \cite{H-geobk}. The substitution isomorphisms associated to 
conformal equivalence classes of spheres with two punctures and local coordinates 
can be constructed directly using the work \cite{H-geobk}. 
But for a grading-restricted M\"{o}bius vertex algebra,
since there are no Virasoro algebra actions, 
we have only functors and isomorphisms associated to 
conformal classes of  sphere with punctures and $SL(2, \C)$ 
local coordinates vanishing 
at punctures. Such a vertex tensor category 
can be called
a M\"{o}bius vertex tensor category. 
The M\"{o}bius vertex tensor category can be obtained easily using
the results in \cite{HLZ8}. }
\end{rema}

We shall use Theorem \ref{1478910-vtc} to prove Theorem \ref{main}.

\renewcommand{\theequation}{\thesection.\arabic{equation}}
\renewcommand{\thethm}{\thesection.\arabic{thm}}
\setcounter{equation}{0}
\setcounter{thm}{0}
\section{$C_{1}$-cofiniteness and intertwining operators}

To apply Theorem \ref{1478910-vtc} to the case that 
$\mathcal{C}$ is the category of $C_{1}$-cofinite grading-restricted 
generalized $V$-modules, we need to verify 
Assumptions 1, 4,  6, 8, 9, 11, and 12. To verify Assumptions 6 and 8,
we need to generalize the results on the differential equations derived by the author for 
 intertwining operators without logarithms 
in \cite{H-diff-eqn}  and observed 
in \cite{HLZ7} to hold also  for logarithmic intertwining operators. 
To verify Assumptions 9, 11, and 12, we need some results of 
Miyamoto in \cite{M} on $C_{1}$-cofinite $\N$-gradable weak 
$V$-modules and also a weak version proved by Miyamoto in \cite{M} 
of the inequality of Nahm in \cite{N}.

In this section, we first recall and give proofs 
of the results of Miyamoto in \cite{M} on $C_{1}$-cofinite $\N$-gradable weak 
$V$-modules. Then  
we generalize the results on the differential equations 
in \cite{H-diff-eqn} by showing that the differential equations
in \cite{H-diff-eqn} are still satisfied even when two of the generalized 
$V$-modules are quasi-finite-dimensional 
but not $C_{1}$-cofinite.  Finally, instead of recalling the weak version of the inequality of Nahm
proved in \cite{M}, we give a proof of the generalization of 
the inequality of Nahm.

Before we give these results, 
we recall again the $C_{1}$-cofiniteness for weak $V$-modules and other types of $V$-modules. 
Let $W$ be a weak $V$-module. We say that $W$ is $C_{1}$-cofinite 
if $\dim W/C_{1}(W)<\infty$, where $C_{1}(W)$ is the subspace of
$W$ spanned by elements of the form $v_{-1}w$ for $v\in V_{+}
=\coprod_{n\in \Z_{+}}V_{(n)}$ and $w\in W$. Since any other classes of 
$V$-modules are all weak modules, this definition also defines $C_{1}$-cofiniteness
for these other classes of $V$-modules. 

We first give a result of Miyamoto in \cite{M} on $C_{1}$-cofinite $\N$-gradable 
weak $V$-modules. This result combines
several results in different places in \cite{M}.
For the reader's convenience, we also provide a proof of this result. 

\begin{prop}\label{C-1-l.g.r}
Let $W$ be a $C_{1}$-cofinite $\N$-gradable weak
$V$-module.
\begin{enumerate}
\item Let $W=\coprod_{n\in \N}W_{\l n\r}$ be a compatible $\N$-grading.
Then for $n\in \N$, $\dim W_{\l n\r} <\infty$. 

\item There exists a finite-dimensional subspace $M$ of $W$
such that $W$ is spanned by elements of the form 
$v^{(1)}_{-1}\cdots v^{(i)}_{-1}w$ for $i\in \N$,
$v^{(1)}, \dots, v^{(i)}\in V$ and $w\in M$. 
In particular, $W$ is finitely generated. 

\item There exists 
$h_{1}, \dots, h_{k}\in \C$
such that they are  not congruent to each other modulo $\Z$ and
$W=\coprod_{i=1}^{k}\coprod_{n\in h_{i}+\N}W_{[n]}$, where 
$W_{[n]}$ for $n\in h_{i}+\N$ are finite-dimensional 
generalized eigenspaces for $L_{W}(0)$ with eigenvalues $n$.
In particular, 
$W$ is a quasi-finite-dimensional generalized $V$-module.
\end{enumerate}
\end{prop}
\begin{proof}
We prove first $\dim W_{\l n\r}<\infty$  for $n\in \N$. 
We use induction on $n$. 
For $v\in V_{(m)}$ and $w\in W_{\l n\r}$, where $m\in \Z_{+}$ and $n\in \N$,
$v_{-1}w\in W_{\l m+n\r}\ne W_{\l 0\r}$.
Hence nonzero elements of $C_{1}(W)$ cannot be in 
$W_{\l 0\r}$. Therefore $\dim W_{\l 0\r}\le \dim W/C_{1}(W)<\infty$ 
since $W$ is $C_{1}$-cofinite. 

Assume that $\dim W_{\l k\r}<\infty$ for $k=0, \dots, n$. 
Then the subspace of $W_{\l n+1\r}$ spanned by elements of the form 
$v_{-1}w$ for $v\in V_{(n+1-k)}$ and $w\in W_{\l k\r}$
for $k=0, \dots, n$
is also finite dimensional since $\dim V_{(n+1-k)}<\infty$
and $\dim W_{\l k\r}<\infty$ for $k=0, \dots, n$. But 
the quotient of $W_{\l n+1\r}$ by this subspace of $W_{\l n+1\r}$
is by definition a subspace of $W/C_{1}(W)$. Since 
$\dim W/C_{1}(W)<\infty$, this quotient is also 
finite dimensional. Then we have  $\dim W_{\l n+1\r}<\infty$.

Since $\dim W_{\l n\r}<\infty$, we also have 
$\dim \coprod_{n=0}^{N}W_{\l n\r}<\infty$ for $N\in \N$. 
But every finite-dimensional 
subspace of $W$ must be in $\coprod_{n=0}^{N}W_{\l n\r}$
for some $N\in \N$. Take a finite-dimensional 
subspace  $M$ of $W$ such that $W=C_{1}(W)+ M$. 
Then there exists $N_{0}\in \N$ such that $M\subset 
\coprod_{n=0}^{N_{0}}W_{\l n\r}$ and thus 
$\coprod_{n=N_{0}+\Z_{+}}W_{\l n\r}
\subset C_{1}(W)$. We can now take $M$ above to be 
$\coprod_{n=0}^{N_{0}}W_{\l n\r}$. So from now on,  
$M=\coprod_{n=0}^{N_{0}}W_{\l n\r}$. 

We are ready to show that 
$W$ is spanned by elements of the form 
$v^{(1)}_{-1}\cdots v^{(i)}_{-1}w$ for $i\in \N$,
$v^{(1)}, \dots, v^{(i)}\in V$ and $w\in M$. Denote the space spanned by
elements of the form above by $\widetilde{W}$. What we want to prove is 
$W=\widetilde{W}$. 

We need to prove that $W_{\l n\r}\subset \widetilde{W}$
for every $n\in \N$. We use induction on $n$.
When $n\le N_{0}$,  $W_{\l n\r}\subset 
M\subset \widetilde{W}$.
Assume that 
for $m< n\in N_{0}+\Z_{+}$, $W_{\l m\r}\subset \widetilde{W}$.
Then $W_{\l n\r}
\subset C_{1}(W)$ is spanned by elements of the form $v_{-1}w$, where 
$v\in V_{(l)}$ and $w\in W_{\l m\r}$ for $l\in \Z_{+}$ and 
$m\in \N$, such that $l+m=n>N$. Since $l>0$, we have 
$m<n$. By induction assumption, $w\in \widetilde{W}$
and thus $v_{-1}w\in \widetilde{W}$. So
$W_{\l (n\r}\subset \widetilde{W}$ and thus Conclusion 2 is true. 

Since $W_{\l n\r}$ for fixed $n\in \N$ is invariant under the action of $L_{W}(0)$
and $\dim W_{\l n\r}<\infty$, $W_{\l n\r}$ can be decomposed 
as a direct sum of generalized eigenspaces of $L_{W}(0)$. 
Hence $W$ can also be decomposed as a 
direct sum of generalized eigenspaces of $L_{W}(0)$. 
In particular, the finite-dimensional invariant subspace 
$M=\coprod_{n=0}^{N_{0}}W_{\l n\r}$ of 
$W$ under $L_{W}(0)$  is a direct sum of finitely many
generalized eigenspaces of $L_{W}(0)$. So there exists 
$h_{1}, \dots, h_{k}\in \C$ and $\widetilde{N}\in \N$ such that  
$h_{1}, \dots, h_{k}$ are not congruent to each other modulo $\Z$ and
$M=\coprod_{i=1}^{k}\coprod_{n=0}^{\widetilde{N}}
M_{[h_{i}+n]}$, where 
$M_{[h_{i}+n]}$ for $i=1, \dots, k$ and $n=0, \dots, \widetilde{N}$
are finite-dimensional 
generalized eigenspaces 
of $L_{W}(0)$ with eigenvalues $h_{i}+n$. 

On the other hand, since 
$W$ is spanned by elements of the form 
$v^{(1)}_{-1}\cdots v^{(i)}_{-1}w$ for $i\in \N$,
$v^{(1)}, \dots, v^{(i)}\in V$ and $w\in M$,
we see that 
$W=\coprod_{i=1}^{k}\coprod_{n\in h_{i}+\N}W_{[n]}$.
where $W_{[n]}$ are finite-dimensional generalized eigenspaces 
of $L_{W}(0)$ with eigenvalues $n$.  Thus $W$ is a quasi-finite-dimensional 
generalized $V$-module. 
\end{proof}

Now we give results involving intertwining operators. 
We refer the reader to \cite{HLZ2} for the precise definition
of (logarithmic) intertwining operators among generalized $V$-modules. 
As we have mentioned and have been doing in Section 2, in this paper,  
we omit  the word ``logarithmic"
to call logarithmic intertwining operators simply intertwining operators,
unless it is necessary to 
emphasize that there are logarithms of the variables. 
Also, we note that (logarithmic) intertwining operators are well defined for 
weak modules. 

We first generalize the results on differential equations 
 of the author in \cite{H-diff-eqn} which is observed in \cite{HLZ7} to 
hold also for logarithmic intertwining operators. In fact, it is  observed 
in Remark 1.7 in \cite{H-diff-eqn} that if one of the $V$-modules is not $C_{1}$-cofinite 
but is generated by a lowerest vector,  the results on
the differential equations in \cite{H-diff-eqn} still hold. We now show that 
the results and proofs on these differential equations in \cite{H-diff-eqn} 
also work if the two generalized $V$-modules placed at $0$ and $\infty$
are quasi-finite-dimensional generalized $V$-modules (not necessarily generated by 
lowerest vectors).

\begin{thm}\label{sys1}
Let $W_{1}$, $W_2$ be 
$C_{1}$-cofinite grading-restricted generalized
$V$-modules and $W_{0}$ and $W_{3}$ quasi-finite-dimensional 
generalized $V$-modules. Then given either the singular point $z_{1}=\infty$, $z_{2}=0$
or the singular point $z_{1}=z_{2}$, $z_{2}=\infty$, 
for $w_{0}\in W_{0}$, $w_{1}\in W_{1}$, $w_{2}\in W_{2}$, $w_{3}\in W_{3}$, there exist
$$a_{k}(z_{1}, z_{2}),
b_{l}(z_{1}, z_{2})\in R=
\mathbb{C}[z_{1}^{\pm 1}, z_{2}^{\pm 1}, (z_{1}-z_{2})^{-1}]$$
for $k=1, \dots, m$ and $l=1, \dots, n$ such that
for lower-bounded generalized $V$-modules $W_{4}$,
intertwining operators $\mathcal{Y}_{1}$ and $\mathcal{Y}_{2}$ of types
$\binom{W_{0}'}{W_{1}W_{4}}$ and $\binom{W_{4}}{W_{2}W_{3}}$,
respectively, the series
\begin{equation}\label{prod-int}
\langle w_{0}, \Y_{1}(w_{1}, z_{1})\Y_{2}(w_{2}, z_{2})w_{3}\rangle
\end{equation}
satisfy the expansions of the system of differential equations
\begin{eqnarray}
\frac{\partial^{m}\varphi}{\partial z_{1}^{m}}+a_{1}(z_{1}, z_{2})
\frac{\partial^{m-1}\varphi}{\partial z_{1}^{m-1}}
+\cdots + a_{m}(z_{1}, z_{2})\varphi&=&0,\label{eqn1}\\
\frac{\partial^{n}\varphi}{\partial z_{2}^{n}}+b_{1}(z_{1}, z_{2})
\frac{\partial^{n-1}\varphi}{\partial z_{2}^{n-1}}
+\cdots + b_{n}(z_{1}, z_{2})\varphi&=&0\label{eqn2}
\end{eqnarray}
of regular singular point at $z_{1}=\infty$, $z_{2}=0$
on the regions $|z_{1}|>|z_{2}|>0$ or at the singular point 
$z_{1}=z_{2}$, $z_{2}=\infty$ on the region $|z_{2}|>|z_{1}-z_{2}|>0$. 
\end{thm}
\begin{proof}
In this proof, we show that calculations and proofs used to derive the differential equations 
in \cite{H-diff-eqn} still work here. So we assume that the reader is familiar with 
the paper \cite{H-diff-eqn}.  

In \cite{H-diff-eqn}, the reason why we need the $C_{1}$-cofiniteness of modules is 
because we need to show that the $R$-module  $T/J$ is finitely generated, where 
$R=\C[z_{1}, z_{1}^{-1}, z_{1}, z_{1}^{-1}, (z_{1}-z_{2})^{-1}]$, 
$T=R\otimes W_{0}\otimes W_{1}\otimes W_{2}\otimes W_{3}$ and $J$ is 
the $R$-submodule of $T$ generated by elements 
$\mathcal{A}(u, w_{0}, w_{1}, w_{2}, w_{3})$, $\mathcal{B}(u, w_{0}, w_{1}, w_{2}, w_{3})$,
$\mathcal{C}(u, w_{0}, w_{1}, w_{2}, w_{3})$ and $\mathcal{D}(u, w_{0}, w_{1}, w_{2}, w_{3})$
given in \cite{H-diff-eqn} for $u\in V$, $w_{0}\in W_{0}$,
$w_{1}\in W_{1}$, $w_{2}\in W_{2}$, $w_{3}\in W_{3}$. 
But this strong result that $T/J$ is finitely generated is in fact not needed in 
\cite{H-diff-eqn}.
When the differential equations are derived  in \cite{H-diff-eqn} for \eqref{prod-int}
with fixed $w_{0}, w_{1}, w_{2}, w_{3}$, only Corollary 1.3 in \cite{H-diff-eqn} 
is needed. Corollary 1.3 in  \cite{H-diff-eqn} states that 
the $R$-submodule of $T/J$ generated by 
elements of the form $[w_{0}\otimes L_{W_{1}}(-1)^{i}w_{1}\otimes 
w_{2}\otimes w_{3}]$ and $[w_{0}\otimes w_{1}\otimes 
L_{W_{2}}(-1)^{j}w_{2}\otimes w_{3}]$ (where we use $[X]$ to denote the 
coset in $T/J$ containing $X$) for $i, j\in \N$ is finitely generated. 
So to derive the differential equations,
we need only show that for fixed $w_{0}, w_{1}, w_{2}, w_{3}$, this 
$R$-submodule of $T/J$ is finitely generated. 

To prove that this $R$-submodule of $T/J$ is finitely generated,
we first prove that a larger $R$-module than this $R$-submodule
 of $T/J$ is finitely generated. Given $N_{0}, N_{3}\in \N$,
let 
$$T_{N_{0}, N_{3}}=R\otimes \left(\coprod_{\Re(n)\le N_{0}}(W_{0})_{[n]}\right)
\otimes W_{1}\otimes W_{2}\otimes \left(\coprod_{\Re(n)\le N_{3}}(W_{3})_{[n]}\right).$$
Then $T_{N_{0}, N_{3}}$ is an $R$-submodule of $T$. Note that elements of the 
form $\mathcal{C}(u, w_{0}, w_{1}, w_{2}, w_{3})$ and 
$\mathcal{D}(u, w_{0}, w_{1}, w_{2}, w_{3})$ for $u\in V$, $w_{0}
\in \coprod_{\Re(n)\le N_{0}}(W_{0})_{[n]}$, $w_{1}\in W_{1}$, $w_{2}\in W_{2}$
and $w_{3}\in \coprod_{\Re(n)\le N_{3}}(W_{3})_{[n]}$ in general might not be in $T_{N_{0}, N_{3}}$.
But elements of the 
form $\mathcal{A}(u, w_{0}, w_{1}, w_{2}, w_{3})$ and 
$\mathcal{B}(u, w_{0}, w_{1}, w_{2}, w_{3})$ for the same $u, w_{0}, w_{1}, w_{2}, w_{3}$
 are always in $T_{N_{0}, N_{3}}$.
Let $J_{N_{0}, N_{3}}$ be the  $R$-submodule of $T_{N_{0}, N_{3}}$ 
spanned by $\mathcal{A}(u, w_{0}, w_{1}, w_{2}, w_{3})$ and 
$\mathcal{B}(u, w_{0}, w_{1}, w_{2}, w_{3})$ for $u\in V$, $w_{0}
\in \coprod_{\Re(n)\le N_{0}}(W_{0})_{[n]}$, $w_{1}\in W_{1}$, $w_{2}\in W_{2}$
and $w_{3}\in \coprod_{\Re(n)\le N_{3}}(W_{3})_{[n]}$.
Since $W_{0}$ and $W_{3}$ are quasi-finite dimensional, 
$\coprod_{\Re(n)\le N_{0}}(W_{0})_{[n]}$ and $\coprod_{\Re(n)\le N_{3}}(W_{3})_{[n]}$
are finite-dimensional. This fact and the same proofs as those of Proposition 1.1 
and Corollary 1.2 in  \cite{H-diff-eqn} shows that $T_{N_{0}, N_{3}}/J_{N_{0}, N_{3}}$ 
is finitely generated. 

For fixed given $w_{0}, w_{1}, w_{2}, w_{3}$, take  $N_{0}, N_{3}\in \N$ such that 
$w_{0}\in \coprod_{\Re(n)\le N_{0}}(W_{0})_{[n]}$ and 
$w_{3}\in \coprod_{\Re(n)\le N_{3}}(W_{3})_{[n]}$. Then 
$T_{N_{0}, N_{3}}/J_{N_{0}, N_{3}}$ is finitely generated. 
Since $R$ is Noetherian, the $R$-submodule of 
$T_{N_{0}, N_{3}}/J_{N_{0}, N_{3}}$  generated by elements of the form 
$[w_{0}\otimes L_{W_{1}}(-1)^{i}w_{1}\otimes 
w_{2}\otimes w_{3}]$ and $[w_{0}\otimes w_{1}\otimes 
L_{W_{2}}(-1)^{j}w_{2}\otimes w_{3}]$  for $i, j\in \N$ is also finitely generated. 

Now the proof of Theorem 1.4 in \cite{H-diff-eqn}  gives the differential equations and the proof 
Theorem 2.3 in \cite{H-diff-eqn}  gives the regularity of 
the singular points.
\end{proof}

The proof of the following result is the same as the proof of Theorem 
\ref{sys1} above:

\begin{thm}\label{sys2}
Let $W_{i}$ for $i=1, \dots, n$ be $C_{1}$-cofinite grading-restricted
generalized $V$-modules and $W_{0}$ and $W_{n+1}$ quasi-finite-dimensional 
generalized $V$-modules.  Then for any $w_{(0)}'\in
W_{0}'$, $w_{(1)}\in W_{1}, \dots, w_{(n+1)}\in W_{n+1}$, there exist
\[
a_{k, \;l}(z_{1}, \dots, z_{n})\in {\mathbb C}[z_{1}^{\pm 1}, \dots,
z_{n}^{\pm 1}, (z_{1}-z_{2})^{-1}, (z_{1}-z_{3})^{-1}, \dots,
(z_{n-1}-z_{n})^{-1}],
\]
for $k=1, \dots, m$ and $l=1, \dots, n,$ such that the following
holds: For any lower-bounded  
generalized $V$-modules $\widetilde{W}_{1}, \dots,
\widetilde{W}_{n-1}$, and any intertwining operators
\[
{\cal Y}_{1}, {\cal Y}_{2}, \dots, {\cal
Y}_{n-1}, {\cal Y}_{n}
\]
of types 
\[
\binom{W_{0}}{
W_{1}\widetilde{W}_{1}}, \binom{\widetilde{W}_{1}}{
W_{2}\widetilde{W}_{2}}, \dots, \binom{\widetilde{W}_{n-2}}{
W_{n-1}\widetilde{W}_{n-1}}, \binom{\widetilde{W}_{n-1}}{
W_{n}W_{n+1}},
\]
respectively, the series
\begin{equation}
\langle w_{(0)}', {\cal Y}_{1}(w_{(1)}, z_{1})\cdots {\cal Y}_{n}(w_{(n)},
z_{n})w_{(n+1)}\rangle
\end{equation}
satisfies the expansion of the system of differential equations
\[
\frac{\partial^{m}\varphi}{\partial z_{l}^{m}}+ \sum_{k=1}^{m}
a_{k,
\;l}(z_{1}, \dots, z_{n}) \frac{\partial^{m-k}\varphi}{\partial
z_{l}^{m-k}}=0,\;\;\;l=1, \dots, n
\]
on the region $|z_{1}|>\cdots >|z_{n}|>0$.  
Moreover, for any set of possible
singular points of the system
\begin{equation}\label{sys-eqns}
\frac{\partial^{m}\varphi}{\partial z_{l}^{m}}+ \sum_{k=1}^{m} a_{k,
\;l}(z_{1}, \dots, z_{n})\frac{\partial^{m-k}\varphi}{\partial
z_{l}^{m-k}}=0,\;\;\;l=1, \dots, n
\end{equation}
such that either $z_i = 0$ or $z_i = \infty$ for some $i$ or $z_i =
z_j$ for some $i \ne j$, $a_{k, \;l}(z_{1}, \dots, z_{n})$ can be
chosen for $k=1, \dots, m$ and $l=1, \dots, n$ so that these singular
points are regular.
\end{thm}

For the next two results, 
we  need the notion of surjectivity of an intertwining operator. 
Let $W_{1}$, $W_{2}$ and $W_{3}$ be weak $V$-modules. 
An intertwining operator $\Y$ of type $\binom{W_{3}}{W_{1}W_{2}}$
is said to be surjective if $W_{3}$ is spanned by 
the coefficients of $\Y(w_{1}, x)w_{2}$ for $w_{1}\in W_{1}$ and $w_{2}\in W_{2}$.
If for a weak $V$-module $W_{3}$, there is an surjective intertwining 
operator of type $\binom{W_{3}}{W_{1}W_{2}}$, we say that 
$W_{3}$ is a {\it weak surjective product of $W_{1}$ and $W_{2}$}. 
If $W_{3}$ is generalized $V$-module or grading-restricted 
generalized $V$-module or other classes of $V$-modules, we also use the terms
{\it generalized surjective product of  $W_{1}$ and $W_{2}$}, 
{\it grading-restricted generalized surjective product  
 of $W_{1}$ and $W_{2}$} and so on.

\begin{prop}\label{int-weak-generalized}
Let $W_{1}$ and $W_{2}$ be generalized $V$-modules and $W_{3}$ 
a weak $V$-module. If $W_{3}$ is a weak surjective product of 
$W_{1}$ and $W_{2}$, then $W_{3}$ is also a generalized $V$-module. 
\end{prop}
\begin{proof}
For $w_{1}\in W_{1}$ and $w_{2}\in W_{2}$, we have
$$\Y(w_{1}, x)w_{2}=\sum_{k=0}^{K}\sum_{n\in \C}\Y_{n, k}(w_{1})
w_{2}x^{-n-1}(\log x)^{k},$$
where $\Y_{n, k}(w_{1})\in \hom(W_{2}, W_{3})$. 
We now take $w_{1}$ and $w_{2}$ to be homogeneous.
We have $L_{W_{1}}(0)=L_{W_{1}}(0)_{S}+L_{W_{1}}(0)_{N}$
and $L_{W_{2}}(0)=L_{W_{2}}(0)_{S}+L_{W_{2}}(0)_{N}$, 
where $L_{W_{1}}(0)_{S}$ and $L_{W_{2}}(0)_{S}$ are the semisimple 
parts of $L_{W_{1}}(0)$ and $L_{W_{2}}(0)$, respectively, and 
$L_{W_{1}}(0)_{N}$ and $L_{W_{2}}(0)_{N}$ are the nilpotent
parts of $L_{W_{1}}(0)$ and $L_{W_{2}}(0)$, respectively.
Then $L_{W_{1}}(0)_{S}w_{1}=(\wt w_{1}) w_{1}$ and 
$L_{W_{2}}(0)_{S}w_{2}=(\wt w_{2}) w_{2}$. 

The commutator formula between $L(0)$ and $\Y$ can be written as 
$$L_{W_{3}}(0)\Y(w_{1}, x)w_{2}=x\frac{d}{dx}\Y(w_{1}, x)w_{2}
+\Y(L_{W_{1}}(0)w_{1}, x)w_{2}+\Y(w_{1}, x)L_{W_{2}}(0)w_{2}.$$
Taking the coefficients of $x^{-n-1}(\log x)^{k}$ on both sides,
using the fact that $w_{1}$ and $w_{2}$ are homogeneous,
and then reorganizing the terms,
we obtain 
\begin{align}\label{W-3-weights}
&(L_{W_{3}}(0)-(\wt w_{1}-n-1+\wt w_{2}))
\Y_{n, k}(w_{1})w_{2}\nn
&\quad =\Y_{n, k}(L_{W_{1}}(0)_{N}w_{1})w_{2}
+\Y_{n, k}(w_{1})L_{W_{2}}(0)_{N}w_{2}
+(k+1)\Y_{n, k+1}(w_{1})w_{2}.
\end{align}
Using  \eqref{W-3-weights} repeatedly, we see that 
$$(L_{W_{3}}(0)-(\wt w_{1}-n-1+\wt w_{2}))^{J}\Y_{n, k}(w_{1})w_{2}$$
for $J\in \Z_{+}$ is a linear comnbination of elements 
of the form 
\begin{equation}\label{W-3-weights-2}
\Y_{n, k+l}(L_{W_{1}}(0)_{N}^{m}w_{1})L_{W_{2}}(0)_{N}^{n}w_{2}
\end{equation}
for $m, n, l\in \N$ such that $m+n+l=J$. When $J$ is sufficiently large, 
at least one of $m, n, l$ must be sufficiently large. 
Since $L_{W_{1}}(0)_{N}$ and $L_{W_{2}}(0)_{N}$ are nilpotent and 
$\Y_{n, k}(w_{1})w_{2}=0$ for $k>K$,  \eqref{W-3-weights-2} must be $0$
when $J$ is sufficiently large.  Hence we obtain 
$$(L_{W_{3}}(0)-(\wt w_{1}-n-1+\wt w_{2}))^{J}\Y_{n, k}(w_{1})w_{2}=0$$
for sufficiently large $J\in \Z_{+}$. Thus we see that 
$\Y_{n, k}(w_{1})w_{2}$ is a generalized eigenvector for 
$L_{W_{3}}(0)$ with eigenvalue $\wt w_{1}-n-1+\wt w_{2}$. 
Since $\Y$ is surjective, $\Y_{n, k}(w_{1})w_{2}$ for $w_{1}\in W_{1}$, 
$w_{2}\in W_{2}$, $n\in \C$ and $0\le k\le K$ span $W_{3}$.
So $W_{3}$ is a generalized $V$-module. 
\end{proof}

The next result is a generalization 
of the inequality of Nahm in \cite{N}. 
This generalization is stronger than a key result obtained later by 
Miyamoto in \cite{M}. Here by generalization, we mean 
that in the theorem below, the inequality holds for any weak surjective product  
of $W_{1}$ and $W_{2}$. We do not need the existence of 
any fusion or tensor product of $W_{1}$ and $W_{2}$
and the surjective  product  
of $W_{1}$ and $W_{2}$ can be a weak $V$-module. 
In fact, this generalization was essentially formulated 
by McRae and Sopin in \cite{MS} except that the surjective  product  
of $W_{1}$ and $W_{2}$ is probably a stronger type of $V$-modules than 
a weak $V$-module. 
It was observed also in \cite{MS} that the inequality can be obtained using
the number of independent solutions of the differential equations 
in \cite{M}. Here we give a full proof of this inequality.

\begin{thm}\label{n-W-1-W-2}
Let $W_{1}$, $W_{2}$ be $\N$-gradable weak $V$-modules. Then 
 for a weak surjective product 
$W_{3}$ of $W_{1}$ and $W_{2}$, 
\begin{equation}\label{nahm-ineq}
\dim (W_{3}/
C_{1}(W_{3}))\le \dim (W_{1}/
C_{1}(W_{1}))\dim (W_{2}/
C_{1}(W_{2})).
\end{equation}
\end{thm}
\begin{proof}
In the case that at least one of $\dim (W_{1}/
C_{1}(W_{1}))$ and $\dim (W_{2}/
C_{1}(W_{2}))$ is $\infty$, \eqref{nahm-ineq} holds. 
So we need only prove \eqref{nahm-ineq} in the case that 
$\dim W_{1}/
C_{1}(W_{1}), \dim W_{2}/
C_{1}(W_{2})<\infty$, that is, in the case that 
$W_{1}$ and $W_{2}$ are $C_{1}$-cofinite $\N$-gradable weak $V$-modules. 
In particular, by conclusions of Proposition \ref{C-1-l.g.r} hold for $W_{1}$ and $W_{2}$. 

Let $M_{1}$ and $M_{2}$ be finite-dimensional graded subspaces of $W_{1}$ and $W_{2}$,
respectively,  such that 
$W_{1}=C_{1}(W_{1})\oplus M_{1}$ and $W_{2}=C_{1}(W_{2})\oplus M_{2}$.
Then $\dim (W_{1}/C_{1}(W_{1}))=\dim M_{1}$ and $\dim (W_{2}/C_{1}(W_{2}))=\dim M_{2}$.

By Proposition \ref{int-weak-generalized}, $W_{3}$
is a generalized $V$-module. In particular, $W_{3}$ is $\C$-graded and 
we have the contragredient $W_{3}'$ of 
$W_{3}$. 
Let $M_{3}$ be a graded subspace of $W_{3}$ 
such that $W_{3}=C_{1}(W_{3})\oplus M_{3}$. Then 
$W_{3}/C_{1}(W_{3})$ 
is isomorphic to $M_{3}$. 
We need only prove 
$\dim M_{3}\le \dim M_{1} \dim M_{2}$. 

Let 
$$C_{1}(W_{3})^{\bot}=\{w_{3}'\in W_{3}'\mid \langle w_{3}', 
C_{1}(W_{3})\rangle=0\}
\subset W_{3}'.$$
and let $M_{3}'$ be the graded dual of $M_{3}$ with respect to the 
$\C$-grading induced from the one on $W_{3}$. 
We define a linear map $r: C_{1}(W_{3})^{\bot}\to M_{3}'$ to 
be the restriction map sending elements of $C_{1}(W_{3})^{\bot}\subset W_{3}'$ 
to  their restrictions to $M_{3}$, that is, 
$(r(w_{3}'))(w_{3})=\langle w_{3}',  w_{3}\rangle$
for $w_{3}'\in C_{1}(W_{3})^{\bot}$ and $w_{3}\in M_{3}$. 
If $r(w_{3}')=0$ for $w_{3}'\in C_{1}(W_{3})^{\bot}$, 
then $\langle w_{3}', w_{3}\rangle=0$ for all $w_{3}\in M_{3}$. 
But by definition, $\langle w_{3}', C_{1}(W_{3})\rangle =0$. Since 
$W_{3}=C_{1}(W_{3})\oplus M_{3}$, we obtain
$\langle w_{3}', w_{3}\rangle=0$ for all $w_{3}\in W_{3}$. 
Thus $w_{3}'=0$, proving that $r$ is injective. 
Given $w_{3}'\in M_{3}'$, we extend it to an element
$\bar{w}_{3}'\in W_{3}'$ by 
$\langle \bar{w}_{3}', \tilde{w}_{3}+w_{3}\rangle=
\langle w_{3}', w_{3}\rangle$ for $\tilde{w}_{3}\in 
C_{1}(W_{3})$ and $w_{3}\in M_{3}$. By definition, 
$\bar{w}_{3}'\in C_{1}(W_{3})^{\bot}$ and $r(\bar{w}_{3}')=w_{3}'$, 
proving $r$ is surjective. 
We have proved that 
$r$ is a linear isomorphism. Therefore we need only prove that
$\dim C_{1}(W_{3})^{\bot}\le \dim M_{1} \dim M_{2}$.

Let $\Y$ be a surjective intertwining operator of type $\binom{W_{3}}{W_{1}W_{2}}$.
We have proved that for homogeneous $w_{1}\in W_{1}$ and $w_{2}\in W_{2}$,
$$\wt \Y_{n, k}(w_{1})w_{2}=\wt w_{1}-n-1+\wt w_{2}.$$ 
Then for homogeneous $w_{3}'\in W_{3}'$,
$$\langle w_{3}',  \Y(w_{1}, x)w_{2}\rangle
=\sum_{k=0}^{K}\langle w_{3}',  \Y_{\swt w_{1}+\swt w_{2}-\swt w_{3}'-1, k}
(w_{1})w_{2}\rangle
x^{-\swt w_{1}-\swt w_{2}+\swt w_{3}'}(\log x)^{k}.$$

Fix $z\in \C^{\times}$. 
We use $\log z$  to denote 
$\log |z|+i\arg z$, where $0\le \arg z<2\pi$. Then
$$\langle w_{3}' , \Y(w_{1}, z)w_{2}\rangle
=\sum_{k=0}^{K}\langle w_{3}',  \Y_{\swt w_{1}+\swt w_{2}-\swt w_{3}'-1, k}
(w_{1})w_{2}\rangle
e^{(-\swt w_{1}-\swt w_{2}+\swt w_{3}')\log z}(\log z)^{k}$$
is well defined. 
We define a linear map 
$f: C_{1}(W_{3})^{\bot}\to (M_{1}\otimes M_{2})^{*}$ by 
$$(f(w_{3}'))(w_{1}\otimes w_{2})
=\langle w_{3}', \Y(w_{1}, z)w_{2}\rangle$$
for $w_{1}\in M_{1}$, $w_{2}\in M_{2}$ and $w_{3}'\in C_{1}(W_{3})^{\bot}$. 
To prove $\dim C_{1}(W_{3})^{\bot}\le \dim M_{1} \dim M_{2}=\dim (M_{1}\otimes M_{2})^{*}$, 
we need only prove that $f$ is injective. 

Assume that $f(w_{3}')=0$ for an element $w_{3}'\in C_{1}(W_{3})^{\bot}$.
Then by the definition of $f$, for $w_{1}\in M_{1}$, $w_{2}\in M_{2}$,
\begin{equation}\label{int-w-1-w-2}
\langle w_{3}', \Y(w_{1}, z)w_{2}\rangle=0
\end{equation}
We now prove \eqref{int-w-1-w-2} for all $w_{1}\in W_{1}$, $w_{2}\in W_{2}$.
By Conclusion 3 of Proposition \ref{C-1-l.g.r}, there exist $h_{1}^{(1)}, \dots, h_{p}^{(1)},
h_{2}^{(1)}, \dots, h_{q}^{(2)}
\in \C$ such that  
$W_{1}=\coprod_{i=1}^{p}\coprod_{n\in h_{i}^{(1)}+\N}(W_{1})_{[n]}$ and 
$W_{2}=\coprod_{j=1}^{q}\coprod_{n\in h_{j}^{(2)}+\N}(W_{2})_{[n]}$. 
For $w_{1}\in \coprod_{i=1}^{p}
(W_{1})_{[h_{i}^{(1)}+n_{1}]}$ and $w_{2}\in \coprod_{j=1}^{q}
(W_{2})_{[h_{j}^{(2)}+n_{2}]}$, we use induction on $n_{1}+n_{2}$ to 
prove \eqref{int-w-1-w-2}.
When $n_{1}+n_{2}=0$, $n_{1}=n_{2}=0$. 
We have $w_{1}\in \coprod_{i=1}^{p}(W_{1})_{[h_{i}^{(1)}]}$ and $w_{1}\in 
 \coprod_{j=1}^{q}(W_{2})_{[h_{j}^{(2)}]}$. We can take $w_{1}$ and $w_{2}$ to be homogeneous.
Then there are $i$ and $j$ such that 
$w_{1}\in 
(W_{1})_{[h_{i}^{(1)}+n_{1}]}$ and $w_{2}\in 
(W_{2})_{[h_{j}^{(2)}+n_{2}]}$.
Since $W_{1}=C_{1}(W_{1})\oplus M_{1}$ and $W_{2}=C_{1}(W_{1})\oplus M_{2}$,
there exist homogeneous $u^{(k)}, v^{(l)}\in V$,
$\tilde{w}_{1}^{(k)} \in W_{1}$, $\tilde{w}_{2}^{(l)} \in W_{2}$, for $k=1, \dots, s$, $l=1, \dots, t$, 
$\tilde{w}_{1}\in M_{1}$, and
$\tilde{w}_{2}\in M_{2}$ such that 
$w_{1}=\sum_{k=1}^{s}u^{(k)}_{-1}\tilde{w}^{(k)}_{1}+\tilde{w}_{1}$ and 
$w_{2}=\sum_{l=1}^{t}v^{(l)}_{-1}\tilde{w}^{(l)}_{2}+\tilde{w}_{2}$.
But for $k=1, \dots, s$, 
$$\Re(\wt \tilde{w}^{(k)}_{1})<\wt u^{(k)}+\Re(\wt \tilde{w}^{(k)}_{1})
=\Re(\wt u^{(k)}_{-1}\tilde{w}^{(k)}_{1})
=\Re(\wt w_{1})=\Re(h_{i}^{(1)}).$$ 
So $u^{(k)}_{-1}\tilde{w}^{(k)}_{1}=0$ and $w_{1}=\tilde{w}_{1}\in M_{1}$.
Similarly, $w_{2}\in M_{2}$. In this case, \eqref{int-w-1-w-2} is true. 

Assume that when $n_{1}+n_{2}<m$, \eqref{int-w-1-w-2} is true. 
In the case $n_{1}+n_{2}=m$,  we still take $w_{1}\in (W_{1})_{[h_{i}^{(1)}+n_{1}]}$ and
still have 
$w_{1}=\sum_{k=1}^{s}u^{(k)}_{-1}\tilde{w}^{(k)}_{1}+\tilde{w}_{1}$
for homogeneous  $u^{(k)}\in V$,
$\tilde{w}_{1}^{(k)} \in W_{1}$ for $k=1, \dots, s$, and $\tilde{w}_{1}\in M_{1}$. 
In this case, 
$$\Re(\wt \tilde{w}^{(k)}_{1})<\wt u^{(k)}+\Re(\wt \tilde{w}^{(k)}_{1})
=\Re(\wt u^{(k)}_{-1}\tilde{w}^{(k)}_{1})
=\Re(\wt w_{1})=\Re(h_{i}+n_{1}).$$ 
So $\tilde{w}^{(k)}_{1}\in \coprod_{m_{1}<n_{1}}
(W_{1})_{[h_{i}^{(1)}+m_{1}]}$
Similarly, we still take $w_{2}\in (W_{2})_{[h_{j}^{(2)}+n_{2}]}$ and still have
$w_{2}=\sum_{l=1}^{t}v^{(l)}_{-1}\tilde{w}^{(l)}_{2}+\tilde{w}_{2}$
for homogeneous  $v^{(l)}\in V$,
$\tilde{w}_{2}^{(l)} \in W_{1}$ for $l=1, \dots, t$, and $\tilde{w}_{2}\in M_{2}$. 
Then $\tilde{w}^{(l)}_{2}\in \coprod_{m_{2}<n_{2}}
(W_{2})_{[h_{j}^{(2)}+m_{2}]}$.
Using the Jacobi identity and the fact that 
$\langle w_{3}', C_{1}(W_{3})\rangle=0$, we see that 
$\langle w_{3}',  \Y(u^{(k)}_{-1}\tilde{w}^{(k)}_{1}, x)w_{2}\rangle$ 
is a linear combination of formal series in $x$ 
of the form $\langle w_{3}', \Y(\tilde{w}^{(k)}_{1}, x)\hat{w}_{2}\rangle$, where 
$\hat{w}_{2}\in \coprod_{\hat{n_{2}}\le n_{2}}
(W_{2})_{[h_{j}^{(2)}+\hat{n}_{2}]}$ with Laurent polynomial 
in $x$ as coefficients. Since $\tilde{w}^{(k)}_{1}\in
 \coprod_{m_{1}<n_{1}}
(W_{1})_{[h_{i}^{(1)}+m_{1}]}$ and $\hat{w}_{2}\in \coprod_{\hat{n_{2}}\le n_{2}}
(W_{2})_{[h_{j}^{(2)}+\hat{n}_{2}]}$ and $m_{1}+\hat{n}_{2}<n_{1}+n_{2}$,
by induction assumption, $\langle w_{3}', \Y(\tilde{w}^{(k)}_{1}, z)\hat{w}_{2}\rangle=0$.
So $\langle w_{3}',  \Y(u^{(k)}_{-1}\tilde{w}^{(k)}_{1}, z)w_{2}\rangle=0$. 
Similarly, 
$\langle w_{3}',  \Y(\tilde{w}_{1}, x)v^{(l)}_{-1}\tilde{w}^{(l)}_{2}\rangle$ 
is a linear combination of formal series in $x$
of the form $\langle w_{3}', \Y(\hat{w}_{1}, x)\tilde{w}^{(l)}_{2}\rangle$, where 
$\hat{w}_{1}\in \coprod_{\hat{n}_{1}\le n_{1}}
(W_{1})_{[h_{i}^{(1)}+\hat{n}_{1}]}$ with Laurent polynomials
in $x$ as coefficients.  Since $\hat{w}_{1}\in \coprod_{\hat{n}_{1}\le n_{1}}
(W_{1})_{[h_{i}^{(1)}+\hat{n}_{1}]}$ and $\tilde{w}^{(l)}_{2}\in \coprod_{m_{2}<n_{2}}
(W_{2})_{[h_{j}^{(2)}+m_{2}]}$ and $\hat{n}_{1}+m_{2}<n_{1}+n_{2}$,
by induction assumption, $\langle w_{3}', \Y(\hat{w}_{1}, z)\tilde{w}^{(l)}_{2}\rangle=0$.
So $\langle w_{3}',  \Y(\tilde{w}_{1}, x)v^{(l)}_{-1}\tilde{w}^{(l)}_{2}\rangle=0$.
Thus 
\begin{align*}
\langle w_{3}', \Y(w_{1}, z)w_{2}\rangle
&=\sum_{k=1}^{s}\langle w_{3}',  \Y(u^{(k)}_{-1}\tilde{w}^{(k)}_{1}, z)w_{2}\rangle
+\langle w_{3}',  \Y(\tilde{w}_{1}, z)w_{2}\rangle\nn
&=\langle w_{3}',  \Y(\tilde{w}_{1}, z)w_{2}\rangle\nn
&=\sum_{l=1}^{t}\langle w_{3}',  \Y(\tilde{w}_{1}, z)v^{(l)}_{-1}\tilde{w}^{(l)}_{2}\rangle
+\langle w_{3}',  \Y(\tilde{w}_{1}, z)v^{(l)}_{-1}\tilde{w}_{2}\rangle\nn
&=\langle w_{3}',  \Y(\tilde{w}_{1}, z)v^{(l)}_{-1}\tilde{w}_{2}\rangle\nn
&=0, 
\end{align*}
proving  \eqref{int-w-1-w-2} 
 for all $w_{1}\in W_{1}$ and 
$w_{2}\in W_{2}$. 

Using the $L(-1)$-derivative property for $\Y$, we have 
$$\langle w_{3}', \Y(w_{1}, \xi)w_{2}\rangle
=\langle w_{3}', \Y(e^{(\xi -z)L_{W_{1}}(-1)}w_{1}, z)w_{2}\rangle=0$$
on the region $|\xi-z|<|z|$. But $\langle w_{3}', \Y(w_{1}, \xi)w_{2}\rangle$
is an analytic function of $\xi$. This analytic function is equal to $0$ 
on a region means that it is equal to $0$ on its domain. So 
we obtain $\langle w_{3}', \Y(w_{1}, \xi)w_{2}\rangle=0$ on the region $\xi\ne 0$. 
Thus
$\langle w_{3}', \Y_{n, k}(w_{1})w_{2}\rangle=0$ for 
$w_{1}\in W_{1}$, 
$w_{2}\in W_{2}$,
$n\in \C$, $k=1, \dots, K$. 
Since $\Y$ is surjective, 
we must have $w_{3}'=0$, proving the injectivity of $f$. 
\end{proof}

Theorem \ref{n-W-1-W-2} has the following 
immediate consequence: 

\begin{cor}\label{image-C-1}
Let $W_{1}$ and $W_{2}$ be $C_{1}$-cofinite grading-restricted generalized
 $V$-modules (or equivalenetly,  $C_{1}$-cofinite $\N$-gradable 
 weak $V$-modules). Then 
 a weak surjective product 
of $W_{1}$ and $W_{2}$ is a $C_{1}$-cofinite 
generalized $V$-module. In particular, a lower-bounded generalized 
surjective product  of $W_{1}$ and $W_{2}$ is a $C_{1}$-cofinite  
grading-restricted generalized $V$-module.
\end{cor}
\begin{proof}
Since $W_{1}$ and $W_{2}$ are generalized $V$-modules, 
by Proposition \ref{int-weak-generalized}, a weak surjective product 
$W_{3}$ of $W_{1}$ and $W_{2}$ is a generalized $V$-module. 
Since $W_{1}$ and $W_{2}$ are $C_{1}$-cofinite, by \eqref{nahm-ineq}
in Theorem \ref{n-W-1-W-2}, 
$\dim (W_{3}/C_{3}(W_{3}))\le \dim (W_{1}/
C_{1}(W_{1}))\dim (W_{2}/
C_{1}(W_{2}))<\infty$. So $W_{3}$ is a $C_{1}$-cofinite 
generalized $V$-module. 

In the case that $W_{3}$ is lower-bounded generalized surjective product  
of $W_{1}$ and $W_{2}$, since $W_{3}$ is also $C_{1}$-cofinite, it is
grading-restricted (in fact quasi-finite-dimensional) 
by Conclusion 3 of Proposition \ref{C-1-l.g.r}.
\end{proof}

\begin{rema}
{\rm Miyamoto proved the following result in \cite{M}: 
Let $W_{1}$ and $W_{2}$ be $C_{1}$-cofinite $\N$-gradable 
weak $V$-modules. Then there exists $d_{W_{1}, W_{2}}\in \N$
such that  for any $\N$-gradable weak surjective product 
$W_{3}$ of $W_{1}$ and $W_{2}$,
$\dim (W_{3}/
C_{1}(W_{3}))\le d_{W_{1}, W_{2}}$. We can obtain this result 
from Theorem \ref{n-W-1-W-2} by taking
$d_{W_{1}, W_{2}}= \dim (W_{1}/
C_{1}(W_{1}))\dim (W_{2}/
C_{1}(W_{2}))$.   Also in this result 
in \cite{M}, $W_{3}$ is an $\N$-gradable weak $V$-nodule 
while in Theorem \ref{n-W-1-W-2}, $W_{3}$ is a weak $V$-module. 
In fact, Creutzig, McRae and Yang 
generalized this result in \cite{M} to the case that $W_{3}$ is a  generalized $V$-module in 
Corollary 2.14 in \cite{CMY2}.  Combining 
Proposition \ref{int-weak-generalized} above and 
Corollary 2.14 in \cite{CMY2}, this result in \cite{M} is immediately 
generalized to the case that $W_{3}$ is a weak $V$-module. 
Also, the proof of this result  in \cite{M}  uses differential equations, which 
are equivalent to the special case  of the differential equations
derived in \cite{H-diff-eqn} for one intertwining operator.
The proof of Corollary 2.14 in \cite{CMY2} uses this result  in \cite{M}
and, in particular,  also uses differential equations  implicitly. 
The proof of Theorem \ref{n-W-1-W-2} does not  use differential equations at all.
On the other hand, as mentioned above, McRae and Sopin observed
in \cite{MS}  that the inequality of Nahm can be obtained using the 
number of independent solutions of the differential equations in 
\cite{M}. }
\end{rema}

\section{Proof of the main theorem}

In this section, we take $\mathcal{C}$ to be the category of $C_{1}$-cofinite 
grading-restricted generalized $V$-modules. 
We verify  in this section that  $\mathcal{C}$ satisfies Assumptions 
1, 4,  6, 8, 9, 11, 12 given in Section 2. Then by Theorem \ref{1478910-vtc},
we obtain Theorem \ref{main}.

For this category $\mathcal{C}$,  Assumption
1 is by definition satisfied. Since a grading-restricted 
M\"{o}bius vertex algebra is $C_{1}$-cofinite, 
Assumption 4 is also satisfied. 
We still need to verify Assumptions 6, 8, 9, 11, and 12. 

We first verify Assumptions 6 and 8. 

\begin{prop}
The convergence and extension property for  products of intertwining operators in $\mathcal{C}$
holds. Moreover, the products of more than two
intertwining operators are absolutely convergent in the corresponding region
and can be analytically extended to multivalued analytic 
functions defined on the region where the complex variables 
in the intertwining operators 
are not equal $0$ and each other.
\end{prop}
\begin{proof}
The first conclusion (Assumption 8) follows from Theorem \ref{sys1}. The proof is the same 
as the proof of Theorem 3.5 in \cite{H-diff-eqn} in the 
case that the intertwining operators do not contain 
logarithms of the variables and the observation in the proof of 
Theorem 11.8 in \cite{HLZ7} that the proof of Theorem 3.5 in \cite{H-diff-eqn}
still works in the case of logarithmic intertwining 
operators. 

The second conclusion (Assumption 6) follows immediately from Theorem \ref{sys2} by
the theory of differential equations of regular singular points.
\end{proof}

\begin{rema}
{\rm In fact, the main difficult part of the proof 
of Theorem 3.5 in \cite{H-diff-eqn} is the proof of the result 
that  in the 
semisimple case, there is no logarithm after the products 
of intertwining operators are analytically extended to 
the region for the iterates of intertwining operators and expanded as
series in powers and logarithms of the variables. 
So the proof in our logarithmic case is in fact simpler. }
\end{rema}

We now verify Assumption 9 using Proposition \ref{C-1-l.g.r}.

\begin{prop}
For any object $W$ in $\mathcal{C}$, the weights of homogeneous 
elements of $W$ are congruent to finitely many complex numbers modulo $\Z$
and there exists $K\in \Z_{+}$ such that 
$L_{W}(0)_{N}^{K}=0$ where $L_{W}(0)_{N}$ is the nilpotent part of 
$L_{W}(0)$. 
\end{prop}
\begin{proof}
The first part of this result follows immediately from 
concludion 3 of Proposition \ref{C-1-l.g.r}.
So we need to prove only the second part. 

By conclusion 2 of Proposition \ref{C-1-l.g.r}, there is a finite-dimensional 
subspace $M$ of $W$
such that $W$ is spanned by elements of the form 
$v^{(1)}_{-1}\cdots v^{(i)}_{-1}w$ for $i\in \N$,
$v^{(1)}, \dots, v^{(i)}\in V$ and $w\in M$. Let $w_{1}, \dots, w_{k}$ be a basis
of $M$. Then there is $K_{i}\in \Z_{+}$ for $i=1, \dots, k$ such that 
$L_{W}(0)_{N}^{K_{i}}w_{i}=0$. Let $K=\max(K_{1}, \dots, K_{k})$. 
We now show that $L_{W}(0)_{N}^{K}=0$.

For any element $w\in M$, by the definition of $K$, we have 
$L_{W}(0)_{N}^{K}w=0$. To prove $L_{W}(0)_{N}^{K}w=0$
for general $w\in W$, we first notice that $L_{W}(0)_{N}$ and 
the vertex operator $Y_{W}(v, x)$ for $v\in V$ commute.
In fact, this is the special case of (4.2) in \cite{H-exist-twisted-mod}
when the twisted module is an untwisted $V$-module (twisted by the identity 
on $V$). In the untwisted case, the vertex operator $Y_{W}(v, x)$ 
does not contain $\log x$ and thus the commutator of 
$L_{W}(0)_{N}$ and $Y_{W}(v, x)$ is $0$. In particular, 
$L_{W}(0)_{N}$ and $v_{-1}$ commute. Thus for an element 
$v^{(1)}_{-1}\cdots v^{(i)}_{-1}w$ for $i\in \N$,
$v^{(1)}, \dots, v^{(i)}\in V$ and $w\in M$, we have 
$$L_{W}(0)_{N}^{K}(v^{(1)}_{-1}\cdots v^{(i)}_{-1}w)
=v^{(1)}_{-1}\cdots v^{(i)}_{-1}L_{W}(0)_{N}^{K}w=0.$$
Since $W$ is spanned by such elements, we obtain $L_{W}(0)_{N}^{K}=0$.
\end{proof}

For Assumption 11, we need only verify 
that $\mathcal{C}$ is closed under $P(z)$-tensor products
for some $z\in \C^{\times}$. 
To verify  this part of Assumption 11 and Assumption 12,
we  prove a general result on  grading-restricted generalized surjective product 
of $W_{1}$ and $W_{2}$ and lower-bounded generalized $V$-modules
contained in  $(W_{1}\otimes W_{2})^{*}$. Then 
Assumptions 11 and 12 are immediate consequences. 
This result can also be proved by modifying the relevant 
results and their proofs in \cite{CJORY}, \cite{CY},  \cite{Mc}
and \cite{CMOY}. 

\begin{thm}\label{gradings-submod}
Let $W_{1}$ and $W_{2}$ be $C_{1}$-cofinite grading-restricted generalized
 $V$-modules.  
\begin{enumerate}
\item  There are $h_{1}, \dots, h_{k}\in \C$, depending only on $W_{1}$ and $W_{2}$,
 such that for any grading-restricted 
generalized surjective product $W_{0}$ of $W_{1}$ and $W_{2}$,  we have
$$W_{0}=\coprod_{i=1}^{k}\coprod_{n\in \N}
(W_{0})_{[h_{i}+n]}.$$ 

\item Let $W$ be a lower-bounded generalized $V$-module 
contained in  $(W_{1}\otimes W_{2})^{*}$ (equipped with the action 
$Y_{P(z)}'(\cdot, x)$ of $V$ and the corresponding actions of $L(-1)$, $L(0)$ and $L(1)$ 
on $(W_{1}\otimes W_{2})^{*}$). 
Then  $W'$ is grading-restricted  generalized
surjective product of $W_{1}$ and $W_{2}$.
In particular,  $W\subset W_{1}\hboxtr_{P(z)}W_{1}$,
\begin{equation}\label{W-lambda-1}
W=\coprod_{i=1}^{k}\coprod_{n\in \N}
(_{[h_{i}+n]},
\end{equation}
where $h_{1}, \dots, h_{k}$ are complex numbers given above, and
$W'$ is $C_{1}$-cofinite, 
\end{enumerate}
\end{thm}
\begin{proof}
We consider the set of $\dim (W_{0}/C_{1}(W_{0}))$ for all
grading-restricted 
generalized surjective product $W_{0}$ of $W_{1}$ and $W_{2}$.
Then by Theorem \ref{n-W-1-W-2},
this set of nonnegative integers 
is bounded from above by $\dim (W_{1}/C_{1}(W_{1}))\dim  (W_{2}/C_{1}(W_{2}))$. 
Hence there must be a maximum of this set. 
Let $W_{0}^{\mathrm{max}}$ be a grading-restricted 
generalized surjective product  of $W_{1}$ and $W_{2}$
such that $\dim W_{0}/C_{1}(W_{0})$
is the maximum of the set. Then by Proposition \ref{C-1-l.g.r}, 
there are 
$h_{1}, \dots, h_{k}\in \C$ such that 
$$W_{0}^{\mathrm{max}}=\coprod_{i=1}^{k}\coprod_{n\in \N}
(W_{0}^{\mathrm{max}})_{[h_{i}+n]}.$$ 

Given any grading-restricted 
generalized surjective product $W_{0}$ of $W_{1}$ and $W_{2}$, we have a 
surjective intertwining operator $\Y$ of type $\binom{W_{0}}{W_{1}W_{2}}$.
For any fixed $z\in \C^{\times}$, we have a $P(z)$-intertwining map 
$I_{\Y, 0}=\Y(\cdot, z)\cdot$ of type $\binom{W_{0}}{W_{1}W_{2}}$.
The $P(z)$-intertwining map $I_{\Y, 0}$ gives 
a $V$-module map $I_{\Y, 0}'$ from $W_{0}'$ to 
the generalized $V$-module $W_{1}\hboxtr_{P(z)}W_{2}$ (see Definition 5.3.1
in \cite{HLZ4}). 
Since the intertwining operator $\Y$ is surjective, this $V$-module 
map is injective.
The image of $W_{0}'$ under this $V$-module map 
is a grading-restricted generalized $V$-submdoule of 
$W_{1}\hboxtr_{P(z)}W_{2}$. This is also true for 
$W_{0}^{\mathrm{max}}$.
Let $W$ be the sum of the image of $W_{0}'$ under
the $V$-module map from $W_{0}'$ to  $W_{1}\hboxtr_{P(z)}W_{2}$ 
and the image of $(W_{0}^{\mathrm{max}})'$ under 
the $V$-module map from $(W_{0}^{\mathrm{max}})'$ to  
$W_{1}\hboxtr_{P(z)}W_{2}$. Then $W$ is 
also a grading-restricted generalized $V$-submodule
of $W_{1}\hboxtr_{P(z)}W_{2}$. Let 
$J: W\to W_{1}\hboxtr_{P(z)}W_{2}$ be the inclusion map. 
Then $J': W_{1}\otimes W_{2}\to \overline{W'}$ 
is a $P(z)$-intertwining map (see Section 5 of \cite{HLZ4}, especially 
Notation 5.25). The $P(z)$-intertwining map $J'$ gives a surjective
intertwining operator $\Y_{J', 0}$ (see (4.18) in \cite{HLZ3})
of type $\binom{W'}{W_{1}W_{2}}$. 
Since $W$ is grading-restricted, $W'$ is also 
grading-restricted. By Corollary \ref{image-C-1},
$W'$ is a $C_{1}$-cofinite grading-restricted 
generalized $V$-module and
$\dim W'/C_{1}(W')\le 
\dim W_{0}^{\mathrm{max}}/C_{1}(W_{0}^{\mathrm{max}})$. 

We now prove 
$$W'=\coprod_{i=1}^{k}\coprod_{n\in \N}
(W')_{[h_{i}+n]}.$$ 
By definition, we have an injective $V$-module map 
from $(W^{\mathrm{max}})'$ to $W$. 
Its adjoint is a surjective $V$-module map 
$f: W'\to W_{0}^{\mathrm{max}}$.
Since $f$ is a $V$-module map, 
$f(C_{1}(W'))\subset C_{1}(W_{0}^{\mathrm{max}})$. 
Then the map $f$ induces a surjective linear map 
$\bar{f}: W'/C_{1}(W')
\to W_{0}^{\mathrm{max}}/C_{1}(W_{0}^{\mathrm{max}})$. In particular, we have 
$\dim (W'/C_{1}(W'))\ge 
\dim (W_{0}^{\mathrm{max}}/C_{1}(W_{0}^{\mathrm{max}}))$. But we already 
have $\dim (W'/C_{1}(W'))\le
\dim (W_{0}^{\mathrm{max}}/C_{1}(W_{0}^{\mathrm{max}}))$. So we obtain 
$$\dim (W'/C_{1}(W'))=
\dim (W_{0}^{\mathrm{max}}/C_{1}(W_{0}^{\mathrm{max}})).$$ 
Thus $\bar{f}$ is injective and 
$\ker \bar{f}=0$. Since $\bar{f}$ is induced from $f$, we 
obtain $\ker f\subset C_{1}(W')$. 

Since $W'$ is a $C_{1}$-cofinite 
grading-restricted 
generalized $V$-module, by Proposition \ref{C-1-l.g.r},
there are $\tilde{h}_{1}, \dots, \tilde{h}_{l}\in \C$ 
such that 
$$W'=\coprod_{i=1}^{l}\coprod_{n\in \N}
(W')_{[\tilde{h}_{i}+n]},$$ 
where $(W')_{[\tilde{h}_{i}+n]}$ for $n\in \N$ 
are finite-dimensional.  Moreover, we can always find such 
$\tilde{h}_{1}, \dots, \tilde{h}_{l}\in \C$ 
such that $(W')_{[\tilde{h}_{i}]}\ne 0$ for 
$i=1, \dots, l$. 
Note that $(W')_{[\tilde{h}_{i}]}$ for $i=1, \dots, l$
cannot be contained in $C_{1}(W')$ and thus cannot be 
in $\ker f$. So for each $i$, nonzero elements of  
$(W')_{[\tilde{h}_{i}]}$ 
are mapped to nonzero homogeneous elements of $W_{0}^{\mathrm{max}}$. 
Hence there exist $1\le j\le k$ and $n\in \N$ such that 
$\tilde{h}_{i}=h_{j}+n$. 
Hence we obtain 
$$W'=\coprod_{i=1}^{k}\coprod_{n\in \N}
(W')_{[h_{i}+n]}.$$
As a consequence, the contragredient $W$ of $W'$
can also be written as 
$$W=\coprod_{i=1}^{k}\coprod_{n\in \N}
W_{[h_{i}+n]}.$$ 
Then as a $V$-submodule of $W$, $W_{0}'$ 
can be written as 
$$W_{0}'=\coprod_{i=1}^{k}\coprod_{n\in \N}
(W_{0}')_{[h_{i}+n]}.$$ 
Thus $W_{0}$ 
can be written as 
$$W_{0}=\coprod_{i=1}^{k}\coprod_{n\in \N}
(W_{0})_{[h_{i}+n]}.$$

We now prove the second part  of the theorem. 
Let $i: W\to (W_{1}\otimes W_{2})^{*}$ be the inclusion map.
Then we have a $P(z)$-intertwining map $i'$ of type $\binom{W'}{W_{1}W_{2}}$
(as in the proof of Theorem \ref{gradings-submod}, see Section 5 of \cite{HLZ4}). 
If $W$ is grading-restricted, then $W'$ is also grading-restricted and, 
as in the proof of the first part above (see 
 (4.18) in \cite{HLZ3}), 
the $P(z)$-intertwining map $i'$ gives a surjective
intertwining operator $\Y_{i', 0}$ of type $\binom{W'}{W_{1}W_{2}}$. Then 
$W'$ is a grading-restricted generalized surjective product of $W_{1}$ and $W_{2}$. 
So to prove the second part of the theorem, we need only 
prove that $W$ is grading-restricted. 

Let $W_{0}$ be the image of $i'$ (meaning the generalized $V$-submodule of 
$W'$ spanned by homogeneous components of elements of the form 
$i'(w_{1}\otimes w_{2})$ for $w_{1}\in W_{1}$ and $w_{2}\in W_{2}$). 
Then we obtain an $P(z)$-intertwining map $I_{0}$ of type $\binom{W_{0}}{W_{1}W_{2}}$.
This $P(z)$-intertwining map $I_{0}$  in turn gives a surjective
intertwining operator $\Y_{I_{0}, 0}$ of type $\binom{W_{0}}{W_{1}W_{2}}$
(as in the proof of the first part above, see (4.18) in \cite{HLZ3}).
Since $W_{0}$ is a generalized $V$-submodule of the lower-bounded generalized $V$-module
$W'$, it is also lower-bounded. 
By Corollary \ref{image-C-1} and Proposition \ref{C-1-l.g.r}, 
$W_{0}$ is a $C_{1}$-cofinite grading-restricted 
generalized $V$-module. 
Then the image $I'_{0}(W_{0}')$ 
of $I'_{0}: W_{0}'\to (W_{1}\otimes W_{2})^{*}$
is in $W_{1}\hboxtr_{P(z)}W_{2}$. If 
$I'(w_{0}')=0$ for $w_{0}'\in W_{0}'$, we have 
$$\langle w_{0}', I_{0}(w_{1}\otimes w_{2})\rangle
= (I'(w_{0}'))(w_{1}\otimes w_{2})=0$$
for all $w_{1}\in W_{1}$ and $w_{2}\in W_{2}$. 
But $W_{0}$ is the image of $I_{0}$. So we obtain $w_{0}'=0$,
showing that $I_{0}'$ is injective. 
Then $I'_{0}(W_{0}')$ is equivalent to $W_{0}'$ as a generalized $V$-module. 
In particular, $I'_{0}(W_{0}')$ is a grading-restricted generalized $V$-submodule
of $W_{1}\hboxtr_{P(z)}W_{2}$ whose contragredient is $C_{1}$-cofinite. 
Let $\chi: W_{0}\to W'$ be the inclusion map.
Then by definition, we have $i'=\chi\circ I_{0}$. So we obtain 
$i''=I_{0}'\circ \chi'$, where $i'': W''\to (W_{1}\otimes W_{2})^{*}$ is the adjoint of 
$i'$. Since $W$ can be embedded into $W''$, we can identify 
$W$ with a generalized $V$-submodule of $W''$. Hence we have
$i=I_{0}'\circ \chi'|_{W}$, where $\chi'|_{W}$ is the restriction of 
$\chi'$ to $W$. Thus $W=i(W)=I_{0}'(\chi'(W))\subset I_{0}'(W_{0}')$. 
Since $I_{0}'(W_{0}')$ is grading restricted, $W$ is also grading-restricted. 
\end{proof}

It is now easy to verify that $\mathcal{C}$ is closed under 
$P(z)$-tensor products for any $z\in \C^{\times}$ and thus Assumption 11 is satisfied. 

\begin{cor}
Let $W_{1}$ and $W_{2}$ be $C_{1}$-cofinite grading-restricted 
generalized $V$-modules. Then the contragredient 
$W_{1}\boxtimes_{P(z)}W_{2}$
of $W_{1}\hboxtr_{P(z)}W_{2}$
is a $C_{1}$-cofinite grading-restricted generalized $V$-module. In particular,
 $\mathcal{C}$ is closed under the $P(z)$-tensor product.
\end{cor}
\begin{proof}
By Proposition \ref{backslash=union},  $W_{1}\hboxtr_{P(z)}W_{2}$
is the sum of  grading-restricted generalized $V$-modules 
contained in $(W_{1}\otimes W_{2})^{*}$
such that their contragredient are objects of $\mathcal{C}$. 
By the second part of Theorem \ref{gradings-submod}, these 
 grading-restricted generalized $V$-modules are of the form
$$W=\coprod_{i=1}^{k}\coprod_{n\in \N}
(W)_{[h_{i}+n]}.$$ 
Thus as a sum of such generalized $V$-modules, $W_{1}\hboxtr_{P(z)}W_{2}$ 
is lower-bounded. Then by the second part of Theorem \ref{gradings-submod} again,
$W_{1}\hboxtr_{P(z)}W_{2}$  itself is also grading restricted 
and its contragredient $W_{1}\boxtimes_{P(z)}W_{2}$ is $C_{1}$-cofinite. 
\end{proof}

We see that Assumption 12 is also satisfied. 

\begin{cor}\label{W-lambda}
Let $W_{1}$ and $W_{2}$ be $C_{1}$-cofinite grading-restricted 
generalized $V$-modules.  Let $\lambda\in (W_{1}\otimes W_{2})^{*}$
be a generalized eigenvector of $L_{P(z)}(0)$
satisfying the compatibility condition.
Assuming that the generalized $V$-module $W_{\lambda}$ 
generated by $\lambda$ is lower bounded. Then 
$W_{\lambda}$ is grading-restricted and $W_{\lambda}'$ is 
$C_{1}$-cofinite. 
\end{cor}
\begin{proof}
This result is the special case of the second part of  
Theorem \ref{gradings-submod} with  $W=W_{\lambda}$
\end{proof}

\noindent {\it Proof of Theorem \ref{main}}. We  have verified Assumptions 1, 4,  6, 8, 9, 11, 12 
for the category $\mathcal{C}$ of $C_{1}$-cofinite grading-restricted generalized $V$-modules. 
By Theorem \ref{1478910-vtc}, we obtain the conclusion of Theorem \ref{main}. 
\hspace*{\fill}\mbox{$\Box$}

\noindent {\small \sc Department of Mathematics, Rutgers University,
110 Frelinghuysen Rd., Piscataway, NJ 08854-8019}

\noindent {\em E-mail address}: yzhuang@math.rutgers.edu

\end{document}